\documentclass[11pt,a4paper]{article}

\usepackage[a4paper,top=3.8cm,bottom=3.8cm,left=3cm,right=3cm]{geometry}%
\usepackage{hyperref}
\usepackage[nottoc,numbib]{tocbibind}
\usepackage[T1]{fontenc}
\usepackage{amsmath,amssymb,mathtools,amsthm,amsopn}
\usepackage{graphicx}
\usepackage[usenames,dvipsnames]{xcolor}

\usepackage{tocloft}
\newcommand{\cut}{\textup{cut}}

\DeclareMathOperator{\Abn}{Abn}
\DeclareMathOperator{\Exp}{Exp}
\DeclareMathOperator{\supp}{supp}

\DeclareMathOperator{\Cut}{Cut}
\DeclareMathOperator{\Conj}{Conj}
\renewcommand{\Im}{\operatorname{Im}}

\usepackage{fancyhdr}
\usepackage{fourier-orns}

\usepackage{enumitem}

\usepackage[dayofweek]{datetime}
\usepackage{mathrsfs}
\renewcommand{\phi}{\varphi}
\DeclareMathOperator{\length}{length}

\newcommand{\step}[1]{\par\medskip\noindent\it#1\rm}
\newcommand{\Scalar}[2]{\left\langle#1,#2\right\rangle}
\newcommand{\scalar}[2]{\langle#1,#2\rangle}
\newcommand{\norm}[1]{\left\Vert#1\right\Vert}

\DeclareMathOperator{\sinc}{sinc}
\newcommand{\abs}[1]{\lvert#1\rvert}

\newcommand{\z}{\zeta}
\renewcommand{\gg}{\mathfrak{g}}
\newcommand{\g}{\gamma}

\newcommand{\wh}{\widehat}
\renewcommand{\H}{\mathbb{H}}
\newcommand{\E}{\mathcal{E}}
\newcommand{\e}{\varepsilon}
\renewcommand{\r}{\rho}

\newcommand{\s}{\sigma}

\newcommand{\la}{\lambda}

\renewcommand{\cal}[1]{\mathcal{#1}}

\newcommand{\wt}{\widetilde}

\newcommand{\ol}{\overline}

\renewcommand{\d}{\delta}

\newcommand{\p}{\partial}

\DeclareMathOperator{\diag}{diag}

\newtheoremstyle{pippo}  % name of the style to be used
  {}       % measure of space to leave above the theorem. E.g.: 3pt
  {}       % measure of space to leave below the theorem. E.g.: 3pt
   {\sffamily}   % name of font to use in the body of the theorem
 {}        % measure of space to indent
  {\bfseries}  % name of head font
  {.}   % punctuation between head and body
  {1ex}       % space after theorem head
  {}           % Manually specify head
\newtheoremstyle{pluto}  {}{}
{\slshape}  {}{\bfseries}  {.} {1ex}    {}

\newtheorem{theorem}{Theorem}[section]
\newtheorem{proposition}[theorem]{Proposition}

\newtheorem{lemma}[theorem]{Lemma}

\newtheorem{corollary}[theorem]{Corollary}

\theoremstyle{pluto}
\newtheorem{definition}[theorem]{Definition}
\newtheorem{remark}[theorem]{Remark}

\usepackage{mdwlist}

\newcommand{\R}{\mathbb{R}}
\newcommand{\G}{\mathbb{G}}
\newcommand{\F}{\mathbb{F}}
\newcommand{\N}{\mathbb{N}}
\newcommand{\C}{\mathbb{C}}
\newcommand{\V}{\mathbb{V}}

\renewcommand{\d}{\delta}
\renewcommand{\t}{\tau}

\renewcommand{\a}{\alpha}
\renewcommand{\b}{\beta}

\DeclareMathOperator{\Span}{span}

\numberwithin{equation}{section}
\numberwithin{figure}{section}
\let\oldbibliography\thebibliography
\renewcommand{\thebibliography}[1]{%
  \oldbibliography{#1}%
  \setlength{\itemsep}{0pt}%
}

\usepackage{titlesec}
\titleformat{\section}{%
\normalfont\large\bfseries}{\thesection.}{1em}{}
\titleformat{\subsection}{%
\normalfont\normalsize\bfseries}{\thesubsection.}{1em}{}
\begin{document}

\title{On the subRiemannian cut locus  in a model 
\\ of free  two-step  Carnot group    
 \thanks{2010 Mathematics Subject Classification. Primary 53C17;
Secondary  49J15.
Key words and Phrases.    Carnot groups, Cut locus,  SubRiemannian geodesic. }}
\author{Annamaria Montanari \and Daniele Morbidelli}

\date{}

\maketitle

 \tableofcontents

 \begin{abstract}
   We characterize the subRiemannian cut locus of the origin in the
free Carnot group of step two with three generators,   giving a new, independent proof of a result by Myasnichenko \cite{Myasnichenko02}.  
We also calculate explicitly the
cut time of any extremal path and the distance from the origin of all points of the
cut locus.   Furthermore,   by using the Hamiltonian approach, we show that the cut time of
strictly normal extremal paths is a smooth explicit function of the initial
  velocity   covector.   Finally, using our previous results, we show that at any cut point the distance has a corner-like singularity.  

 \end{abstract}

% \begin{abstract}
%    We consider a family $\H:= \{X_j = f_j\cdot\nabla: j=1,\dots, m\}$ of $C^1$ vector 
% fields in
% $\R^n$ and let $s\in\N$. We assume that for all $p\in\{1.\dots, s\}$ and $j_1, \dots, j_p\in \{1, \dots, m\}$
% the \emph{horizontal derivatives} $X_{j_1}X_{x_2}\cdots X_{j_{p-1}}f_{j_p}$
% exist and are Lipschitz continuous with respect to the control distance defined by~$\H$.
% Then we show that different notions of commutators agree.
% This involves an accurate analysis of some algebraic identities involving nested commutators which seem to have an independent interest. 
% 

% We apply these  results to  discuss the regularity properties of a class of \emph{almost exponential maps} associated with the system of vector fields.
% Ultimately, we get the proof of the doubling property and the Poincar\'e inequality for H\"ormander vector fields under an intrinsic ``horizontal regularity'' assumption on their coefficients.
% \end{abstract}
% 

%\newpage

% 
% \begin{figure} [ht]
%  \label{figura1}
% \centerline{\includegraphics{palladue.1} } \caption{ The sphere of
% radius $r$.}
% \end{figure}
% \newpage
% 
% 
% \begin{figure} [ht]
%  \label{figura2}
% \centerline{\includegraphics{sandro.1} } \caption{ Funzione $R$.}
% \end{figure}
%% \begin{figure} [ht]
%  \label{figura2}
%  \centerline{\includegraphics{sandro.1} } \caption{ Funzione $R$.}
% \end{figure}

% \begin{figure} [ht]
%  \label{figura2}
% \centerline{\includegraphics{funzione_P.1} } \caption{ Funzione $R$.}
% \end{figure}

\section{Introduction and statement of the results}

One of the most interesting aspects of subRiemannian analysis is the study of the cut locus of a given 
 distance.
It is well known that  a correct understanding of the cut locus is crucial in problems concerning subRiemannian 
optimal transport (see \cite{AmbrosioRigot04,AgrachevLee09,FigalliRifford10} and analysis of the subelliptic heat kernel (see \cite{BarilariBoscainNeel12}).

 In this paper, we introduce a very explicit technique which provides the calculation of the subRiemannian cut locus 
of the origin in the free Carnot group of step two with three generators. 
We are also able to explicitly calculate  the cut time of any extremal path and  the distance from the origin of all points of the cut locus.  Furthermore, we show the presence of Lipschitz singularities at any point of the cut locus.

% \color{blue} Our calculations are completely different and independent from Myasnichenko results
% 
% In this paper we exhibit a  very   explicit description of the subRiemannian cut locus 
% of the origin in the free Carnot group of step two with three generators. \color{blue} Our calculations are completely different and independent from Myasnichenko results
% We also calculate explicitly the cut time of any extremal path and  the distance from the origin of all points of the cut locus. \color{blue} Finally, we show the presence of Lipschitz singularities at any point of the cut locus. \color{black}

 To state our results, let us introduce the free Carnot group of step two with three generators.     Let  
  $\V$ be a three dimensional vector space equipped with an inner product~ $\scalar{\cdot}{\cdot}$.    Consider the 
  six-dimensional linear space 
 $ \F_3 :=\V\times \wedge^2 \V$ with the Lie group law
 \begin{equation}\label{liberone}
(x,t)\cdot(\xi,\tau)=\Bigl(x+\xi, t+\tau+\frac 12 x\wedge  \xi\Bigr),
\end{equation} 
where $(x,t)$ and $(\xi,\tau)\in \V\times \wedge^2 \V$.
Let $d$ be the left invariant subRiemannian distance   defined on the Lie group $(\F_3,\cdot)$ by the data $(\V,\scalar{\cdot}{\cdot})$  
and 
denote by $\Cut_0$ the cut locus of the origin  (see Section~\ref{preliminare} for detailed explanations).

The first result of this  paper can be stated as follows:
\begin{theorem}\label{teo1.1}
 Let $\F_3$ be the free two step Carnot group   with three generators.  Then 
 \begin{equation}\label{unodue} 
  \Cut_0  =\big\{ (x, t): t\neq 0\in\wedge^2 \V\text{ and } x \perp \supp t \big\}.
 \end{equation} 
\end{theorem}
Here, if $y,y'\in\V$, given the bivector 
$t=y\wedge y'\in\wedge^2 \V$, we denote by $\supp(t)=\Span \{y,y'\}$ its support (see the discussion in Section~\ref{pensaci}). 
  The set $\Cut_0$ is a smooth four dimensional submanifold of $\F_3$ (see Remark \ref{quattrocchi}). 
 The cut locus of any point $(x,t)\neq (0,0)$ can be easily obtained from $\Cut_{0}$ by  group translation.

\begin{remark}Some comments on this theorem are in order.  
\begin{enumerate}[nosep]
\item   The cut locus $\Cut_0$
has been already described by Myasnichenko \cite{Myasnichenko02} with a technique based on the analysis of conjugate points. Our technique is completely independent from the analysis of the conjugate locus and it allows us to get several more information on the subRiemannian distance (see the theorems below).  
 \item 
In \cite[Lemma~2.11]{RiffordTrelat09}, Rifford and Tr\'elat prove that $\Cut_0\cup\{0\}$ 
is closed in absence of  abnormal minimizers. 
In our model, which contains abnormal minimizers, it turns out that 
 $\Cut_0\cup\{0\}$  is not closed. Indeed, points of 
 the form 
 $(x,0)\in \V\times\wedge^2 \V$ (which belong 
 to the abnormal set, see \cite{LeDonneMontgomeryOttazziPansuVittone}) are never cut points,
 but 
 they belong to the closure of $\Cut_0$.

% \item  In Carnot groups of step 2 and corank 1  (see \cite{AgrachevBarilariBoscain12}) \color{black} 
%  and in Heisenberg-type groups
%    (see \cite{AutenriedMolina16}),  \color{black}
%  $\Cut_0$ is  contained in the center of the group. Here, as in the corank two cases discussed in 
%  \cite{BarilariBoscainGauthier12}, $\Cut_0$ is not contained in the center.

 \item If we think of the Heisenberg group  as the free two step   Carnot  group with two generators $\F_{2,2}$ 
 (i.e.~$\operatorname{dim}V=2$), 
 then
 formula  \eqref{unodue} reduces as expected to the known fact $\Cut_0=\{(0,t)\in \V\times\wedge^2 \V\sim 
 \V\times \R\}$. Formula \eqref{unodue} is dimension free and one could conjecture that such
 structure of the cut locus holds in any dimension. We plan to come back to these generalizations in a 
 subsequent project.\footnote{  After the submission of this paper, a precise conjecture for the cut locus on the free group in higher dimension has been formulated by Rizzi and Serres \cite{RizziSerres16}.} 

 \end{enumerate}
 \end{remark}

 Fine properties of  subRiemannian length-minimizing curves and of the related distance in Carnot groups of step two are studied by 
 \cite{AgrachevBarilariBoscain12} and  \cite{BarilariBoscainGauthier12} in the corank one and two. The case of h-type groups is discussed in  \cite{AutenriedMolina16}.
%  \cite{ButtSachkovBhatti}. 
 Analysis in step three examples has been performed in \cite{AgrachevBonnard97}, for the Martinet case, and in  
 \cite{ArdentovSachkov11,ArdentovSachkov14} for the Engel group. 
We also
 mention the very recent paper    \cite{BarilariBoscainNeel16}, where a detailed discussion of the cut locus in the biHeisenberg group is performed.   
 Outside of the setting of Carnot groups, we mention
 \cite{Sachkov11} and  \cite{PodobryaevSachkov15}.      Finally, we quote the results in 
\cite{ArcozziFerrari07} on the cut locus from surfaces in the Heisenberg group.

Our further results involve a calculation of the cut time of any extremal curve and  an explicit computation of the distance from the origin of points $(x,t)\in \Cut_0$. 
To describe such result, we recall that,   up to a reparametrization,   length-minimizers for the subRiemannian length have the form 
$s\mapsto \gamma(s)=(x(s),t(s))$, where $\dot x(s)=:u(s)$ has the form
\begin{equation}\label{UUU}
u(s)=a\cos(2\phi s)+b\sin(2\phi s) +z
\end{equation}
  and   $\dot t(s) =\frac 12 x(s)\wedge u(s).  $        Here   the vectors 
  $a,b,z\in\R^3 $   must be pairwise perpendicular,  $\abs{a}=\abs{b}>0$ and  $\phi $ can be assumed to be nonnegative.   The vector $z$ can possibly vanish (this corresponds to a   curve  which lives in a ``Heisenberg subgroup'').   See the discussion in Section \ref{jjmm} for a precise description. We will  calculate the cut time  of any such extremal curve in terms of the following explicit functions:
\begin{equation}\label{uvuesse} 
\begin{aligned}
S(\theta)
% &:=s_\phi
:=\frac{\sin\theta}{\theta},\qquad
% \\
U(\theta)
% & :=u_\theta
:=\frac{\theta-\sin\theta\cos\theta}{4\theta^2},\qquad
% \\
V(\theta) 
% &:=V_\theta
:=\frac{\sin\theta-\theta\cos\theta}{2\theta^2}.
\end{aligned}
\end{equation} 
  More precisely, letting 
\begin{equation}\label{RRR}
Q(\theta)=-\frac{U(\theta)S(\theta)}{V(\theta)},
\end{equation}
  and denoting by   $\phi_1\in \mathopen]\pi, \frac 32\pi\mathclose[$  the first positive zero of the function $V$ in \eqref{uvuesse},   we will show that the function $Q$ is a strictly increasing bijection from 
 $[\pi,\phi_1\mathclose[$ to  $[Q(\pi),Q(\phi_1-)\mathclose[=[ 0,+\infty\mathclose[$ (see Lemma \ref{quoziente} and the plot in figure \ref{figura2}).   Then we will prove the following theorem. 
\begin{theorem}\label{teo1.3}
  Let $a,b,z\in \V$ be orthogonal vectors with $\abs{a}=\abs{b}>0$ and let $\phi> 0$. Consider the control 
\[
 u(s)=a\cos(2\phi s)+b\sin(2\phi s)+ z
\]
and let 
 $s\mapsto \gamma(s,a,b,z, 2\phi )$ be the corresponding curve. Then 
 $\gamma(\cdot ,a,b,z, 2\phi )$ 
 minimizes length up to the time 
\begin{equation*}
t_\cut (a,b,z,\phi)=
\frac{Q^{-1} (|z|^2/|a|^2)} {\phi}.
\end{equation*}
\end{theorem}
Observe that, if we choose $z=0$ we get $t_\cut (a,b,0,\phi)=\frac{\pi}{\phi}$, which is the familiar case of the Heisenberg group. 
 However, it is interesting to remark that the   cut time displays  the following ``discontinuous'' behavior.  If $a,b,z,\phi$ are fixed with  $a,b,z\neq 0$ pairwise orthogonal, $\abs{a}=\abs{b}$ and $\phi>0$, it turns out that  
\[
t_\cut (\e a, \e b,z, \phi) =\frac{1}{\phi}Q^{-1}\Bigl(\frac{|z|^2}{\e^2|a|^2}\Bigr)
\to
  \frac{Q^{-1}(+\infty)}{\phi}=
 \frac{\phi_1}{\phi},
\]  
as $\e\to 0^+$. However, the cut time of the limit curve $\gamma(s)=(sz, 0)$ is $+\infty$.
Note that such limit curve is abnormal (see Section \ref{preliminare}). 
  Using the Hamiltonian  approach,  we will express the cut time as an explicit function  of the initial velocity covector (see Section~\ref{sec5}). 
% A couple of related questions concerning the cut time in higher dimensional free group $\F_m$ with $m\geq 4$ are:\footnote{non lo metterei questo}
% \begin{enumerate}[nosep,label=(Q\arabic*)]
% \item  Is it true or not that $t_\cut <\infty$ for all strictly normal extremal curve in $\F_m$?
% 
% \item   Are there abnormal minimizers in $\F_m$  such that $t_\cut<\infty$? 
% \end{enumerate}

Our   further  result involves the explicit computation of the distance from the origin  and an arbitrary point of the cut locus.
To state our result we introduce the real valued functions 
\begin{equation*}
   P(\theta): =  -\frac{S(\theta) }{V(\theta)}\sqrt{\frac{W(\theta)}{U(\theta)}}\quad\text{and}\quad
   R(\theta):=\frac{1-S(\theta)^2}{\sqrt{U(\theta)W(\theta)}},
\end{equation*} 
where
 $W(\theta):=U(\theta)-S(\theta)V(\theta)$.
The function $R$ is well defined and positive for $\theta>0$, 
because $W(\theta)>0$ for all $\theta>0$  (see Lemma~\ref{vudop}).  Furthermore, we will show that if $\phi_1\in\mathopen]\pi,\frac 32\pi\mathclose[$ denotes 
the first positive zero of the function $V$ in \eqref{uvuesse}, then   $P:\mathopen[\pi,\phi_1\mathclose[\to \mathopen[
0,+\infty\mathclose[$ is an increasing bijection. We denote by $P^{-1}$ the inverse function of $P\bigr|_{\mathopen[\pi,\phi_1\mathclose[}$.  
Then we will prove the following formula:
\begin{theorem}\label{secondoteorema} 
 Let $(x,t)\in \Cut_0$. Then 
 \begin{equation}\label{distanziatore} 
  d\bigl((0,0), (x,t)\bigr)^2
  =\abs{x}^2 + 
%   \frac{1-s(\theta)}{\sqrt{u(\theta)^2-
%   s(\theta)V(\theta)u(\theta)}}
R(\theta)
\abs{t},\quad\text{ where }
  \theta=P^{-1}\Bigl(\frac{\abs{x}^2}{\abs{t}}\Bigr)\in\mathopen[\pi,\phi_1\mathclose[.
 \end{equation} 
\end{theorem}
 We will see  that  the equation $P(\theta)=\frac{\abs{x}^2}{\abs{t}}$ has a countable family of solutions $\theta_k\in \mathopen]k\pi,\phi_k\mathclose[,$
where $\phi_k  \in\bigl]k\pi,(k+\frac 12)\pi \bigr[$ is the $k$-th positive solution of $\tan \theta=\theta.$  See the discussion after \eqref{definiscopi}. 
In Theorem \ref{secondoteorema} we prove that the minimizing choice  is the smallest one $\theta_1$.

If we specialize formula \eqref{distanziatore} to the case $x=0$,   we must take
$\theta=P^{-1}(0)=\pi$ and then $S(\theta)=0$, $U(\theta)=\frac{1}{4\pi}$ and we find 
\[
 d((0,0),(0,t))^2=4\pi\abs{t},
\]
which is the familiar  formula for the distance between the origin and a point in the center in the Heisenberg group (for a complete description of geodesics in Heisenberg group see \cite{Monti00}).  
The value $\theta=\phi_1$ is never achieved.

  Our final result concerns the regularity of the subRiemannian distance from the origin in the cut locus.   Observe that  weak regularity properties (like semiconcavity) have been discussed in \cite{CannarsaRifford,FigalliRifford10,MontanariMorbidelli15,LeDonneNicolussi15}. 
Concerning smooth regularity properties, it is well known that the subRiemannian distance from a point~$0$
cannot be smooth in any pointed neighborhood of~$0$. However,  it is known 
 that the subRiemannian distance is smooth outside of the \emph{abnormal set} and of the cut locus. 
In our model, it turns out that the abnormal set has the form $ \Abn_0=\V\times 0_{\wedge^2 \V}$ (see Section \ref{abnormalia}) and it has been proved that the distance from the origin $d$ is not differentiable at  any  abnormal point  (see \cite{MontanariMorbidelli15} for precise estimates). As a byproduct of our Theorem~\ref{unmaintheorem} it is easy to see
 that $d$ is not differentiable at any point of $\Cut_0$ (see Remark~\ref{multipli}).
% Observe also that~$\overline{\Cut_0}=\Cut_0\cup\Abn_0$. 
 In the present paper we obtain more precise estimates  using the description  of the  profile for the subRiemannian spheres
(Corollary \ref{proffilo}).    Indeed, we are able to   exhibit the presence of a kind of  two-dimensional corner-like singularities  at cut points. Our  estimates also ensure that the distance cannot be semiconvex at any point of the cut locus.   
\begin{theorem}
 \label{referee}
 Let $(\ol x, \ol t)\in \Cut_0$. Then, there is $C>0$ and a $C^1$-smooth two-dimensional submanifold $\Sigma$  containing $(\ol x, \ol t)$ such that 
 \begin{equation}
  d(x,t)\leq d(\ol x, \ol t)-C\abs{(x,t)-(\ol x, \ol t)}\quad \text{for all $(x,t)\in\Sigma$.}
 \end{equation} 
 Furthermore, for any open neighborhood $\Omega$  of $(\ol x,\ol t)$ we have the estimate
 \begin{equation}\label{dossero} 
  \inf_{(x,t), (x+h, t+k),(x-h, t-k)\in\Omega}\frac{d(x+h, t+k)+ d(x-h, t-k)-2d( x,  t)}{h^2+k^2}=-\infty.
 \end{equation} 
\end{theorem}
Estimate \eqref{dossero} gives a  positive answer in our model to 
a question raised by Figalli and Rifford (see the \emph{Open problem} at p~145-146 in  \cite{FigalliRifford10}). 
% 
% \footnote{Tagliamo questa parte? Therefore, one expects that  
% the distance is smooth in the open  set $\F_3\setminus(\Cut_0\cup\Abn_0)$. A rigorous proof of such statement
% involves a detailed discussion of the subRiemannian exponential and of conjugate points. 
% We plan to come back to such aspects in a further research.} 

The paper is structured as follows.
Section \ref{preliminare} contains preliminaries. 
In Section \ref{sec3} we prove  the inclusion $\Sigma\subset \operatorname{Cut}_0$ and Theorem \ref{secondoteorema}. 
 In Subsection \ref{quattrouno}    we prove Theorem~\ref{teo1.3} and we conclude the proof of Theorem \ref{teo1.1}   (inclusion $\Cut_0\subset\Sigma$).   In Subsection \ref{quattrocentodue} we describe the profile of subRiemannian spheres.    Section \ref{sec5} contains some remarks on the Hamiltonian point of view.  
Using such approach we show that in the strictly normal case, the cut time of extremal paths is a smooth explicit function of  the initial covector.   
In Section \ref{sfogliatella} we state and prove Theorem \ref{referee}, on the singularity of the subRiemannian distance at cut points.

\section{Preliminaries}\label{preliminare} 

\subsection{Bivectors}\label{pensaci} 
 Let $\V=\Span\{e_1,\dots, e_m\}$ be a linear space.
 Define  $\wedge^2\V:=\Span\{e_j\wedge e_k:1\leq j<k\leq m\}$. 
Given  two vectors $x,y\in \V$, the elementary bivector  $t=x\wedge y \in\wedge^2 \V$  can be expanded as
\[
 x\wedge y=\sum_{j }(x_j e_j)\wedge\sum_k  (y_k e_k)= \sum_{1\leq j<k\leq m}(x_jy_k-x_ky_j)e_j\wedge e_k
 =:\sum_{1\leq j<k\leq m}(x\wedge y)_{jk}e_j\wedge e_k.
\]
If $e_1,\dots, e_m$ is an orthonormal  basis   with respect to some inner product $\scalar{\cdot}{\cdot}$ on $\V$, 
we define the related product  on   $\wedge^2 \V$ by requiring that the basis 
$e_j\wedge e_k$ with $1\leq j<k\leq m$ is orthonormal. It turns out that, on elementary bivectors,
\begin{equation}\label{giroditalia} 
 \scalar{x\wedge y}{\xi\wedge \eta}=\scalar{x}{\xi}\scalar{y}{\eta}
 -\scalar{x}{\eta}\scalar{y}{\xi}
 \quad\text{for all } x,y,\xi,\eta\in\R^m.
\end{equation}
The inner product $\scalar{z}{\zeta}$ can be extended linearly to general bivectors.
% $z=\sum_{a=1}^n x_a \wedge y_a $ and $\zeta=\sum_{\a=1}^\nu \xi_\a\wedge \eta_\a$, for any $x_a,y_a,\xi_\a,\eta_\a\in\R^m$.
% Note that if $\R^m=V\oplus W$ decomposes as a sum with $V\perp W$ and  we  choose orthonormal 
% bases 
% $v_1,\dots, v_p$ of $V$
% and $w_1,\dots w_q$ of $W$, it turns out that  the family  $\{ v_j\wedge v_k, v_j\wedge w_\a,w_\a\wedge w_\beta: 1\leq j<k\leq p,\quad  1\leq \a<\beta\leq q\}$ is an orthonormal  basis of   $\wedge^2\R^m$ and ultimately
% the three terms in the decomposition 
% \begin{equation}
% \label{joeyy} 
%  \wedge^2\R^m=\wedge^2 V\oplus (V\wedge W)\oplus\wedge^2W 
% \end{equation} 
% are pairwise orthogonal. Here and hereafter we are keeping the short notation $V\wedge W:=\Span\{v\wedge
% w:v\in V,w\in W\}$.
% 
 
 Concerning the three dimensional case, $\operatorname{dim}\V=3$, it is well known that any bivector $t\in\wedge^2 \V$ can be written in the elementary form
$  t=x\wedge y,$
for suitable $x,y\in \V$. Although the choice of $x$ and $y$ is not unique, 
the subspace $\Span\{x,y\}$ does not depend on such choice and it is called \emph{support} of the bivector.  See the discussion in \cite{MontanariMorbidelli15}, where the higher dimensional case is treated.
 
% Let $M=-M^T\in\R^{m\times m}$ be a  skew-symmetric matrix of rank $2p\leq m$. By spectral theory, there are $p$ two-dimensional pairwise orthogonal subspaces $V_1,\dots V_p$,
% $p$ positive numbers $\la_1,\dots, \la_p>0$  and a corresponding orthonormal basis $v_h,v_h^\perp $ of each $V_h$ such that 
% \[
%  Mv_h=  \la_h v_h^\perp\quad\text{and }\quad M v_h^\perp = -\la_h v_h
%  \quad\text{ for all $h=1,\dots,p$.}
% \]
% In other words, we can write
% $
%  Mx=\sum_{h=1}^p \la_h \bigl(\scalar{x}{v_h} v_h^\perp - \scalar{x}{v_h^\perp}v_h\bigr)
% $.

Finally, it is easy to check that if    $ \V $ is a finite dimensional vector space   and if   $x_0$ and $y_0$ are independent in $\V$, 
then   $x\wedge y=x_0 \wedge y_0 $  implies that  $x,y\in\Span\{x_0,y_0\}.$

% Prodotto interno $\scalar{x\wedge y}{\xi\wedge \eta}$
% Copiare dal vecchio paper.
% \begin{enumerate}[nosep]
%  \item 
% \end{itemize}

\subsection{The free group \texorpdfstring{$\F_3$}{F3}} 
\label{jjmm}Let $\V$ be a three-dimensional vector
space with an inner product $\scalar{\cdot}{\cdot}$. Denote by $(x,t)$ variables in 
$\V\times\wedge^2 \V$. If $e_1, e_2, e_3$ is an orthonormal  basis of $\V$, then 
$x=x_1 e_1 + x_2 e_2 + x_3 e_3\sim(x_1, x_2, x_3)$ and $t=t_{12}e_1\wedge e_2+ t_{13}e_1\wedge e_3+
t_{23} e_2\wedge e_3\sim (t_{12},t_{13}, t_{23})\in\R^3$. Introduce the law
\begin{equation}\label{liberaccio}
(x,t)\cdot(\xi,\tau)=\Bigl(x+\xi, t+\tau+\frac 12 x\wedge  \xi\Bigr),
\end{equation}  
Denote by $\F_3$ the Lie group $(\V\times \wedge^2 \V, \cdot )$.
%with operation \eqref{liberaccio}. 

  Fix an orthonormal basis $e_1,e_2,e_3$ of $V$ and corresponding coordinates $(x_1,x_2,x_3,t_{12}, 
  t_{13}, t_{23})$ in $\V\times\wedge^2 \V$. Define the family of three vector fields
\begin{equation*}
% \begin{aligned}
 X_1
%  &
=\p_1-\frac {x_2}{2}\p_{12}-\frac {x_3}{2}\p_{13},\quad 
%  \\ 
 X_2 
%  &
 = \p_2 +\frac {x_1}{2} \p_{12}-\frac{x_3}{2}\p_{23},\quad 
% \\
X_3 
% &
= \p_3+\frac {x_1}{2}\p_{13}+\frac {x_2}{2}\p_{23}.
% \end{aligned}
\end{equation*}
Note the commutation relations $[X_j, X_k]=\p_{jk}$ for all $j,k$ such that $1\leq j<k\leq 3$. 
The vector fields $X_1,X_2,X_3$ are  homogeneous of degree $1$   with respect to the family of dilations $(\delta_r)_{r>0}$ defined by
\begin{equation}\label{dilation}
 \delta_r(x,t)=(rx,r^2 t)\qquad\forall \quad (x,t)\in V\times \wedge^2 V.
\end{equation}
Namely, we have $X_j(f\circ\delta_r)(x,t)=r X_jf(\delta_r(x,t))$ for all function $f$.

  A  path $\gamma\in W^{1,2}((0,T),\F_3)$ 
is said to be \emph{horizontal} if 
there is a \emph{control} $u\in L^2((0,T),\R^3)$ such that
we can write $\dot\gamma(s)=\sum_{j=1}^3u_j(s) X_j(\gamma(s)) $ for a.e.~$s\in(0,T)$. 
The  length of the horizontal path $\gamma$ is  $\length(\gamma):=\int_0^T\abs{u(t)}dt $. If $\gamma$ is arclength, then we have    $\length(\gamma)= \sqrt{T} (\int_0^T\abs{u(t)}^2dt  )^{1/2}$. 
Given points $(\wh x,\wh t)$ and $(x,t)$, define 
$d((\wh x,\wh t),(x,t))=\inf\bigl\{ \length(\gamma)\}$, where the infimum is taken among all horizontal curves $\gamma$ such that $\gamma(0)=(\wh x,\wh t)$ and $\gamma(T)=(x,t)$.  
It is well known that for any pair of points $(x,t),(\wh x,\wh t) \in\F_3$ the infimum is a minimum, i.e. there is a length-minimizing path.
 %  From  and 
Constant speed length-minimizing paths in $\F_3$ are curves $\gamma$ associated with controls   
(which we call \emph{extremal})  of the form
 \begin{equation}\label{fuoricontrollo} 
  u(s)=a\cos(\la s)+b\sin(\la s) +z,
 \end{equation}
where $a,b,z\in \V$ is an \emph{admissible triple}. By \emph{admissible triple} we mean an ordered triple of vectors $a,b,z\in\R^3$,  where the vectors  are pairwise orthogonal and such that
$\abs{a}=\abs{b} \gvertneqq 0$. 
These facts are well known (see \cite{AgrachevGentileLerario,LeDonneMontgomeryOttazziPansuVittone}). Here  we refer to the self-contained discussion in 
\cite[Section~3.1.3]{MontanariMorbidelli15}.
% , where  we know that extremal paths in the group $\F_3$ 
 Without loss of generality,  we can always assume that  $\lambda\geq 0$. The case $\lambda <0$ can be 
 recovered by an adjustment of the sign of $b$.
   In many computations below, controls will be written in the form $a\cos(2\phi s)
 +b\sin (2\phi s) + z$, and we will use 
%  we will let $\lambda=2\phi $ and we shall use 
 the variable  $\phi $. 

 The curve $s\mapsto \gamma(s)=(x(s),t(s))$ corresponding to the control $u$ in \eqref{fuoricontrollo} is obtained 
 integrating the ODE 
\begin{equation}\label{malafemmina} 
\begin{aligned}
 &\dot x = u
 ,\qquad 
 \dot t =\frac 12 x\wedge u
\end{aligned}
\end{equation}
with initial data $(x(0), t(0))=(0,0)$.
An elementary computation gives   the general form 
of extremal curves  
% \footnote{ho cancellato $= \gamma(s, a,b,z,2\phi)$}
\begin{equation}\label{jolo} 
\begin{aligned}
 \gamma(s, a,b,z,\la) & %= \gamma(s, a,b,z,2\phi)%
 = (x(s,a,b,z,\la),t(s ,a,b,z,\la))
\\
&= \bigg( \frac{\sin(\la s)}{\la}a+\frac{1-\cos(\la s)}{\la}b+sz,
\\  & 
 \qquad 
% \Bigl\{
\frac{\la s-\sin(\la s)} {2\la^2}
% \Bigr\} 
a\wedge b
+
% \Bigl\{
\frac{2(1-\cos(\la s))  - \la s\sin(\la s)}{2\la^2} 
% \Bigr\}
a\wedge z
\\& \qquad \qquad \qquad \qquad \qquad+
\frac{\la s(1+\cos(\la s)) - 2\sin(\la s)}{2\la^2}
b\wedge z 
\bigg).
\end{aligned}
\end{equation} 
 For integration of similar systems in higher dimension, see \cite{MonroyMeneses06}.
On sufficiently small intervals, 
 $\gamma$ is a length-minimizer (see \cite{LiuSussmann}, \cite{Rifford14} and \cite{AgrachevBarilariBoscain}).
Under linear change of parameter and  dilation in \eqref{dilation},  $\gamma$ behaves as follows
\begin{align}
\label{riparametrizzo}  \gamma(\mu s, a,b,z,\la)=\gamma (s,\mu a,\mu b,\mu z,\mu\la)
\qquad\text{for all $\mu>0 $, }
 \\
\label{dilato} 
 \gamma(  s, ra,rb,rz,\la)=\delta_r\gamma(s,a,b,z,\la)\qquad\text{for any  $r>0$.}
 \end{align}
  Observe that  $\gamma(s,a,b,z,\lambda)$ tends smoothly to $\gamma(s,a,b,z,0)=(a+z,0)$, as $\lambda\to 0$. 
   Finally, observe the following  rotation invariance property of the distance: 
  if $M\in O(3)$, then  $d(x, y\wedge z)= d( Mx,  My\wedge Mz)$ for all $x,y,z\in \R^3$.
\subsection{Abnormal curves in \texorpdfstring{$\F_3$}{Sigma}}\label{abnormalia} 
Let $u\in L^2((0,1),\R^3)$. It is easy to see that the solution  $s\mapsto\gamma_u(s)=:(x_u(s),t_u(s))$    of the ODE \eqref{malafemmina} with initial data $\gamma_u(0)=(0,0)$ is defined globally on $[0,1]$ for any choice of $u\in L^2((0,1),\R^3) $.
Denote by $E(u)$ such solution at time $s=1$. The map $E: L^2\to \F_3 $ is called \emph{end point} map.

Let $u(s)$ be an extremal control of the form 
% It has been shown in \cite{AgrachevGentileLerario} and \cite{MontanariMorbidelli15} that if a control 
% $u\in L^2((0,1),\R^3)$ is minimizing (i.e.~$\length(\gamma_u)=d(\gamma_u(0),\gamma_u(1))$ \color{black}, then $u$ must have the form \eqref{fuoricontrollo}. Let $u(s)$ be an extremal control,  i.e.
%  a function $u(s)$ of the form 
 \eqref{fuoricontrollo}.  By definition, the control $u$ is abnormal if the differential  $dE(u):L^2((0,1),\R^3)\to \F_3$ is not onto.   The corresponding curve $\gamma(\cdot,a,b,z,\lambda)$
 is called \emph{abnormal extremal curve}.  
 Denote by $\Abn_0$ the set of points  $(x,t)=\gamma(1)$, where $\gamma:[0,1]\to \R^3$ is an abnormal extremal curve with $\gamma(0)=(0,0)$.   
 From the discussion in \cite{LeDonneMontgomeryOttazziPansuVittone} (see also \cite[Section~3.1.3]{MontanariMorbidelli15}), we know  
 that 
\[
\Abn_0=\{(x,0)\in \V\times \wedge^2 \V\}=\V\times 0_{\wedge^2 \V}.
\]

\subsection{Cut time}
\begin{definition} 
Let   $a,b,z$ be an admissible triple, fix $\phi>0$
and let  $s\mapsto \gamma(s,a,b,z,2\phi) $  be the curve defined in \eqref{jolo}.  Define
\begin{equation}
\begin{aligned} t_\textup{cut}(a,b,z,\phi)&=\inf\Bigl\{\ol s\geq 0: 
 \text{  $  \gamma(\cdot ,a,b,z,2\phi) $ does not minimize length on $[0,\ol s]$
% e  $\gamma(\ol s,a,b,z,2\phi)$}
% \\& \text{ e che ha lunghezza minore di $\gamma|_{[0,,\ol s]}$
}
\Bigr\}
\\&=\sup  \Bigl\{\ol s>0 : 
 \text{  $   \gamma(\cdot ,a,b,z,2\phi) $ is a length minimizer on $[0,\ol s]$}
\Bigr\}.
\end{aligned}
\end{equation} 
\end{definition}
 It is well known that $t_\cut>0$ for all such curve. Moreover,   if $t_\cut<\infty$, then the supremum is a maximum.

Standard invariance properties give the following general form for the cut time. 
\begin{proposition}\label{ticutto}  There exists a function
$h_\cut:[0,+\infty\mathclose [ 
\to  \mathopen  ] 0,+\infty] $ such that
\begin{equation}\label{ticatto} 
 t_\textup{cut}(a,b,z,\phi)=\frac{h_\cut(\abs{z}/\abs{a})}{ \phi }
\end{equation}
for all admissible triple $a,b,z$ and $\phi>0$.
\end{proposition}
Note that we do not need to define $h_\cut(\infty)$ because the case $z\neq 0$ and $a=b=0$ corresponds to a constant control $u(s)=z$. Such kind of control is included in the class of admissible controls by choosing  $\phi=0$. In such case it is easy to see that   the corresponding curve minimizes globally, i.e. $t_\cut=+\infty$.
  Observe finally that $h_\cut$ can never vanish, by a classical result in control theory (see \cite{LiuSussmann}, \cite{Rifford14} and \cite{AgrachevBarilariBoscain}) 

One of the main result of our paper is the calculation of the function $h_\cut$. Interestingly,    it turns out that
the function $h_\cut$ is bounded.

\begin{proof}   
In order to show that the cut time of the curve corresponding to the control $u(s)=a\cos(2\phi s)+
b\sin(2\phi s)+ z$ depends on $\abs{a},\abs{z}$ and 
$\phi $, 
it suffices to identify $\V$ with $\R^3=\Span\{e_1,e_2,e_3\}$ where $e_1,e_2,e_3$ are orthonormal and $a=r e_1$, $b= r e_2$ and $z =\rho e_3$. Then the control $u$ becomes
$
 u(s)=\cos( 2\phi s)r e_1+\sin(2\phi  s) re_2+\rho e_3
$,
and it is clear that the cut time depends on $r,\rho$ and $ \phi$.

Let now $\phi=1 $, i.e.~consider the path
$ \sigma\mapsto \gamma(\sigma, a,b, z,2)$. Let $\ol\sigma(r,\rho)$ its cut time, where   $r=\abs{a}$ and $\rho=\abs{z}$. Taking an arbitrary $\phi>0$,  the riparametrization invariance \eqref{riparametrizzo}
shows that the path
\[
 s\mapsto\gamma(s,a,b,z,2\phi)=\gamma\Bigl(\phi s,\frac{a}{\phi} ,\frac{b}{\phi}, \frac{z}{\phi},2\Bigr)
\]
is length minimizing up to the time $s=\dfrac{\ol\sigma( r/\phi,\rho/\phi)}{\phi }$.  
By the dilation invariance \eqref{dilato}, we also have  
\begin{equation*}
\ol\sigma( r/\phi, \rho/\phi)=\ol\sigma(r,\rho)=\ol\sigma(1,\rho/r)=:h_\cut(\rho/r)
\in\left]0,+\infty\right].
\end{equation*}
Thus the cut time has the form   \eqref{ticatto}, as required.
\end{proof}
We know that     $h(0)=\pi$, because in such case we are in Heisenberg subgroups of the form 
$\Span\{(a,0),(b,0),(0,a\wedge b)\}$. We will show that  $h_\cut(\mu)\to\phi_1$, as $\mu\to+\infty$, where $\phi_1\in \mathopen]\pi,\frac 32\pi\mathclose[$ denotes the first positive solution of the equation 
$\tan\phi=\phi$.

Define for $\phi=\la/2\geq 0$  and $a,b,z$ admissible triple
\begin{equation}\label{FFF} 
\begin{aligned}
 F(a,b,z, \phi):& =\gamma(1,a,b,z,2\phi).
\end{aligned}
\end{equation} 
% Introducing the notation
% \begin{equation}
% \begin{aligned}
% s(\phi)
% % &:=s_\phi
% :=\frac{\sin\phi}{\phi},\qquad
% % \\
% u(\phi)
% % & :=u_\phi
% :=\frac{\phi-\sin\phi\cos\phi}{4\phi^2},\qquad
% % \\
% v(\phi) 
% % &:=v_\phi
% :=\frac{\sin\phi-\phi\cos\phi}{2\phi^2}
% \end{aligned}
% \end{equation} 
Recall the  definition of the functions $S,U,V$ introduced in \eqref{uvuesse}
\begin{equation*}
\begin{aligned}
S(\phi)
% &:=s_\phi
:=\frac{\sin\phi}{\phi},\qquad
% \\
U(\phi)
% & :=u_\theta
:=\frac{\phi -\sin\phi\cos\phi}{4\phi^2},\qquad
% \\
V(\phi) 
% &:=V_\theta
:=\frac{\sin\phi-\phi\cos\phi}{2\phi^2}.
\end{aligned}
\end{equation*} 
After some elementary calculations, it turns out that
\begin{equation}\label{laeffes} 
 F(a,b,z,\phi)=\Bigl(S(\phi)(a\cos\phi+b\sin\phi)+ z\,,\;
 U(\phi)a\wedge b+ V(\phi)(a\sin\phi -b\cos\phi)
 \wedge z\Bigr).
\end{equation} 
\begin{remark}\label{cambiovariabile} After the orthogonal change of variable 
\begin{equation}\label{cambiovaluta} 
 a'=a\sin \phi-b\cos\phi\quad\text{and}\qquad b'=a\cos\phi+b\sin\phi
\end{equation}
we have
\begin{equation*}
 \begin{aligned}
F(a,b,z,\phi)=G(a', b', z,\phi), \quad\text{ where}
\\
G(\a,\b, \z ,\phi): =
\bigl(S(\phi) \b  +\z ,\a\wedge(U(\phi)\b+V(\phi)\z)\bigr)\end{aligned}
\end{equation*} 
for all admissible $\a,\b,z$ and $\phi>0$. Note that $a',b',z$ are admissible if and only if $a,b,z$ are admissible. Moreover, $\abs{a'}=\abs{a}$, $\abs{b'}=\abs{b}$  and $a'\wedge b'=a\wedge b$.  
% Actually the function $G$ would appear naturally if we would integrate the control  $u(s)=a\cos(  2\phi \color{black} s)+b\sin(  2\phi \color{black} s)+z$ on the interval $[-\frac 12,\frac 12]$, instead of $[0,1]$.
%
%
%for all admissible $\a,\b,z$. Note that $a',b',z$ are admissible if and only if $a,b,z$ are admissible. Moreover, $\abs{a'}=\abs{a}$, $\abs{b'}=\abs{b}$  and $a'\wedge b'=a\wedge b$.  Actually the function $G$ would appear naturally if we would integrate the control  $u(s)=a\cos(\la s)+b\sin(\la s)$ on the interval $[-\frac 12,\frac 12]$, instead of $[0,1]$.
%for all admissible $\a,\b,z$. Note that $a',b',z$ are admissible if and only if $a,b,z$ are admissible. Moreover, $\abs{a'}=\abs{a}$, $\abs{b'}=\abs{b}$  and $a'\wedge b'=a\wedge b$.  Actually the function $G$ would appear naturally if we would integrate the control  $u(s)=a\cos(\la s)+b\sin(\la s)$ on the interval $[-\frac 12,\frac 12]$, instead of $[0,1]$.
\end{remark}

% 
% 
%  \\
% \label{derivodue}  u'(\theta)&=\frac{\cos\theta}{\theta}V(\theta)
%  \\
% \label{derivotre}   \text{  se serve...}
% \end{align}
% 
% \end{subequations}

\section{Extremal paths and the set \texorpdfstring{$\Sigma$}{Sigma}}\label{sec3}
Define the set $\Sigma\subset\R^3\times\wedge^2\R^3$ as
\begin{equation}
\begin{aligned}
 \Sigma: &=\Bigl\{ (x,t)\in\R^3\times\wedge^2\R^3: x\perp\operatorname{supp} t\text{ and $t\neq 0$}\Bigr\}
 \\&= \{ (x,y \wedge y' ): x,y,y' \in\R^3 \text{ where
  $ y$ and $  y' $  are independent and $x\perp\{y,y' \}$}\}.
\end{aligned}
\end{equation} 
  Here and hereafter, without loss of generality, we may assume that $\V=\R^3$ and
$\scalar{\cdot}{\cdot}$ is the Euclidean inner product. 
% where without loss of generality we assume that $\scalar{y}{y^\perp}=0$ and $\abs{y}=\abs{y``\perp}>0$.
We will show in this section that    $\Sigma\subset \Cut_0$. The proof of the opposite inclusion will be achieved at the end of Subsection \ref{quattrouno}.

Define the function 
\begin{equation} \label{vudoppio}
 W(\theta)=U(\theta)-S(\theta)V(\theta)=\frac{\theta^2+ \theta\sin\theta\cos\theta-2\sin^2\theta}{4\theta^3},
\end{equation}
where $S,U$ and $V$ are defined in \eqref{uvuesse}.

\begin{lemma}\label{vudop} 
$ W(\theta)>0$ for all $\theta>0$. 
\end{lemma}
\begin{proof}
Let $q(\theta)=\theta^2+\theta\sin\theta\cos\theta-2\sin^2\theta$, 
for all  $\theta>0$ we get   
$
 q(\theta)>\theta^2- \theta-2$ which is positive if $\theta>2$.
For $\theta\in(0,2)$, write $2\theta=x $ and 
$
 q(\theta)=\frac 14(x^2+x\sin(x) - 4+4\cos(x))=:\frac 14p(x).
$
Differentiating we find $p(0)=p'(0)=p''(0)=0$  and 
$
 p'''(x)=\sin(x)-x\cos x
$,  which is positive 
for all $x\in(0,\phi_1)$, where $\phi_1>4$ is the first positive solution of $\tan(x)=x$. Thus the function is increasing in $ \theta \in(0,  \phi_1/ 2 )\supset(0,2)$ and the proof is concluded.
\end{proof}

Let
$P:\mathopen]0,+\infty\mathclose[\setminus\bigcup_{k\in\N}\{\phi_k\}$
\begin{equation}\label{definiscopi} 
 P(\theta): =  -\frac{S(\theta) }{V(\theta)}\sqrt{\frac{W(\theta)}{U(\theta)}}
\end{equation} 
where $\phi_k\in\mathopen]k\pi, (k+\frac 12)\pi\mathclose[$ denotes the $k$-th positive solution of $V(\theta)=0$ (i.e.~$\tan(\theta)=\theta$). Note that $P(k\pi)=0$ and $P(\phi_k-)=+\infty$ for all $k\in\N$. We will show that $P$ is strictly increasing in $\mathopen]\pi,\phi_1\mathclose[$. Actually the same happens in any interval 
$\mathopen]k\pi, \phi_k\mathclose[$ for $k\geq 2$, but we do not need such property. A plot of $P$ is exhibited in Figure \ref{figura1}.

\subsection{Characterization of extremal controls connecting the origin with \texorpdfstring{$\Sigma$}{Sigma}}
Next we find all extremal curves (minimizing or not) which connect the origin with a point $(x,t)\in\Sigma$.  Without loss of generality we always write a vector $t\in\wedge^2 \R^3$ in the form 
$t=y\wedge y^\perp$ for a suitable pair of orthogonal
vectors with equal length. In such case, it turns out that $\abs{t}=\abs{y}^2$ (see \eqref{giroditalia}).  
\begin{theorem} [Characterization of extremal paths passing for a given point of $\Sigma$]
\label{unmaintheorem} 
Let $(x,y\wedge y^\perp)\in\Sigma$, where $x,y,y^\perp$ are pairwise orthogonal and $\abs{y}=\abs{y^\perp}>0$.
 Then the following facts hold.  
\begin{enumerate}[nosep,label=(\roman*)]
 \item  \label{itemuno} Let $\theta >0$ be any solution of the equation
\begin{equation}\label{equazioniamo} 
 P(\theta)
 =\frac{\abs{x}^2}{\abs{y}^2}.
\end{equation} 
Let $\sigma\in\R$ be an arbitrary parameter and choose the corresponding vectors
\begin{equation}\label{corresponding} 
 \left\{
 \begin{aligned}
& \a' : =\a'_\s:=\frac{1}{(U_\theta W_\theta)^{1/4}}(y\sin\s -y^\perp\cos\s)  
\\&
\b' :=\b'_\s:= \frac{(U_\theta W_\theta)^{1/4}}{W_\theta}(y\cos\s+y^\perp\sin\s)-\frac{V_\theta}{W_\theta}x
\\& 
\z: =\z_\s:= \frac{U_\theta}{W_\theta}x-\frac{S_\theta}{W_\theta}(U_\theta W_\theta)^{1/4}(y\cos\s+y^\perp\sin\s) 
 \end{aligned}
\right.
\end{equation} 
Then the triple $\a'_\s,\b'_\s,\z_\s$ is admissible and the path 
\begin{equation}\label{camminassimo} 
 s\mapsto \gamma(s, \a'_\s\sin\theta +
 \b'_\s\cos\theta , \b'_\s\sin\theta-\a'_\s\cos\theta,\z_\s ,2\theta)=:
\gamma_\s(s)
\end{equation} 
 satisfies $\gamma_\s(1)=(x,y\wedge y^\perp)$
(the function $\gamma(s,a,b,z,2\phi) $ is defined in \eqref{jolo}).

\item \label{itemdue} If $\theta,\a',\b',\z$  satisfy \eqref{equazioniamo} and \eqref{corresponding}, for all $\sigma\in\R $ the length of the path  $\gamma_\s$ in \eqref{camminassimo} does not depend on $\sigma$ and is 
\begin{equation*}
 \length^2(\gamma_\s|_{[0,1]}) = \abs{x}^2+\frac{1-S(\theta)^2}{\sqrt{U(\theta)W(\theta)\,}}\abs{y}^2.
\end{equation*} 

\item \label{itemtre} If $\theta,\a',\b',\z$  satisfy \eqref{equazioniamo} and \eqref{corresponding}, for all $\sigma\in\R $ we have
\begin{equation}
 \frac{\abs{\z}^2}{\abs{\b'}^2}=\frac{\abs{\z}^2}{\abs{\a'}^2}=  -\frac{U(\theta)S(\theta)}{V(\theta)}
 =:Q(\theta).
\end{equation} 

\item \label{itemquattro} Let  $\gamma$ be an extremal path on $[0,1]$ which satisfies $\gamma(0)=0$ and $\gamma(1)=(x,y\wedge y^\perp)$.    Then there is $\theta>0 $ satisfying   
\eqref{equazioniamo}, there is $\sigma\in\R$ such that $\gamma$ has the form \eqref{camminassimo}, where $\a'_\s,\b'_\s$ and $\z_\s$ are the vectors appearing  in \eqref{corresponding}.

\end{enumerate}

\end{theorem}
  Item \ref{itemuno} and \ref{itemquattro} give the characterization of all extremal  paths connecting the origin with points of $\Sigma$. The length in \ref{itemdue} is needed to discuss their minimizing properties and ratio $Q(\theta)$ appearing in \ref{itemtre} is crucial in the calculation of the cut time. 

\begin{figure} [ht]
 \label{figura1}
%  \centerline{\includegraphics{funzione_s.1} } \caption{ File funzione s. Funzione $s$ su $[0,2\pi]$}
 \includegraphics{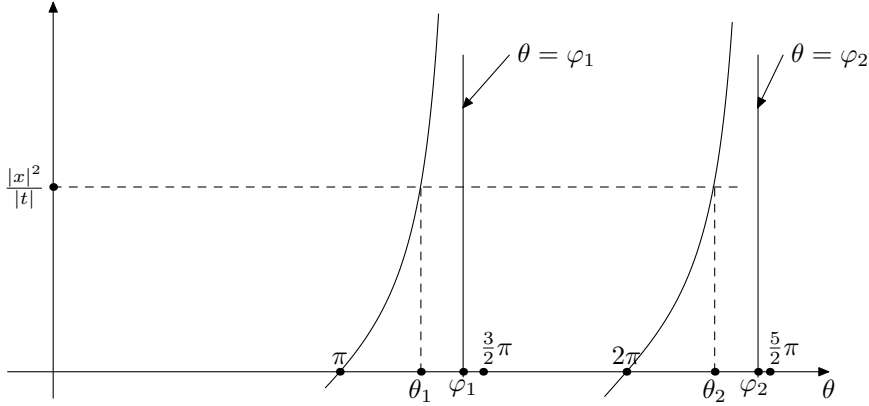}  
% \centerline{\includegraphics{function_P_gimp.jpg}  }
\caption{A plot of the positive part of the function $P$  for $\theta<\phi_2$ with a representation of $\theta_1$ and $\theta_2$, the first two (of the infinitely many)  solutions  of the equation $P(\theta)=\frac{\abs{x}^2}{\abs{t}}$.}
\end{figure}

\begin{remark}
If we take a curve  $\gamma_\s(s)=(x_\s(s),t_\s(s))$ and  we rearrange the choice of $y,y^\perp$ to ensure that $\a'$ and $\b'$ take the simple form
\begin{equation}\label{riarrangeria}
\a'=\frac{y}{(U_\theta W_\theta)^{1/4}}\qquad 
\b'=\frac{(U_\theta W_\theta)^{1/4}}{W_\theta}y^\perp -\frac{V_\theta}{W_\theta}x
\qquad 
\z = \frac{U_\theta}{W_\theta}x-\frac{S_\theta}{W_\theta}(U_\theta W_\theta)^{1/4}
 y^\perp 
,\end{equation}
we discover that 
the corresponding control  $u(s) $  takes the  simple   form 
\begin{equation}
\begin{aligned}
 u(s)&=
% a\cos(2\theta s) + b\sin(2\theta s)+\z \\&=
(\a'\sin\theta+\b'\cos\theta)\cos(2\theta s)+(\b'\sin\theta - \a'\cos\theta)\sin(2\theta s)
+\z
 \\&=\a'\sin(\theta(1-2s))+\b'\cos (\theta(1-2s))+\z
\end{aligned}
\end{equation} 
where $\a',\b',\z$ are the vectors in \eqref{riarrangeria} and   $s\in[0,1]$. 
\end{remark}

 \begin{proof}[Proof of Theorem \ref{unmaintheorem}]
We start with the proof of part  \ref{itemquattro} of the statement. Let $(x,y\wedge y^\perp)\in\Sigma$ and assume that 
 $x,y,y^\perp$ are pairwise orthogonal,  $\abs{y}=\abs{y^\perp}>0$,   while $x$ may possibly vanish.  Let
$s\mapsto \gamma(s,\a,\b,\z,2\theta)$ be the extremal path defined in \eqref{jolo} and assume that  $\gamma(1)=(x,y\wedge y^\perp)$. We must solve the system $F(\a,\b,\z,\theta)=(x, y\wedge y^\perp)$, where $F$ appears in \eqref{FFF} and  
\eqref{laeffes} 
    and the triple $\a,\b,\z$ must be admissible.
Write for brevity $U_\theta, V_\theta, S_\theta$ instead of $U(\theta),V(\theta), S(\theta)$.
Observe also that  $\theta$ cannot be $0$, because $F(\a,\b,\z,0)=(a+z,0)$ cannot belong to $\Sigma$.
Thus assume that $\theta$
is strictly positive (we can always  exclude the case $\theta<0$, see the discussion after \eqref{fuoricontrollo}).   Using the explicit form of $F$ we find the system 
\begin{equation*}
\left\{
\begin{aligned}
 &S_\theta(\a\cos\theta+\b\sin\theta)+\z=x
 \\&
 U_\theta \a\wedge\b+V_\theta(\a\sin\theta - \b\cos\theta)\wedge\z = y\wedge  y^\perp .
\end{aligned}
\right.              \end{equation*}
Making the change of variable
\begin{equation}\label{alexis} 
 \a'=\a\sin\theta-
 \b\cos\theta\quad\text{and}\quad \b'=\a\cos\theta+\b\sin\theta,
 \end{equation} 
(see Remark \ref{cambiovariabile}), 
%  $\a',\b',\z$ is an admissible triple iff $\a,\b,\z$ is an admissible triple. 
% Moreover,  $\a'\wedge\b'=\a\wedge\b$. Then 
we get 
\begin{equation}\label{lapponia} 
\left\{
\begin{aligned}
 &S_\theta\b'+\z=x
 \\&
 \a'\wedge\bigl( U_\theta \b' +V_\theta\z) = y\wedge  y^\perp.               
\end{aligned}
\right.              \end{equation}
Therefore $\a'$ and $U_\theta \b' +V_\theta\z$ belong to $\Span\{y,y^\perp\}$.   Furthermore they must be orthogonal, by admissibility.  This means that we can write
\begin{align}
\label{laddo}  U_\theta \b' +V_\theta\z &=\xi y+\eta y^\perp
 \\
 \label{laddo2} 
 \a' &= q\eta y-q\xi y^\perp,
\end{align}
where $q\neq 0$ and $\xi,\eta\in\R$ satisfy $\xi^2+\eta^2\neq 0$.
The first line of \eqref{lapponia} and \eqref{laddo} are linear in $\b'$ and $\zeta$. Solving, we get
\begin{equation}\label{lapponie} 
 \b'= \frac{\xi}{W_\theta}y+\frac{\eta}{W_\theta}y^\perp-\frac{V_\theta}{W_\theta}x \quad\text{and}\quad
\qquad  \z= 
 \frac{U_\theta}{W_\theta}x -\frac{S_\theta\xi}{W_\theta}y-\frac{S_\theta\eta}{W_\theta}y^\perp,
\end{equation} 
where $W_\theta :=U_\theta-S_\theta V_\theta $ cannot vanish, by Lemma \ref{vudop}. 
On the vectors $\a'$ in \eqref{lapponia} and $ \b',\z$ in \eqref{lapponie}, we must require that
% 
% (A1): 
$\abs{\a'}=\abs{\b'}$. Moreover it must be
% 
% (A2): 
$\scalar{\b'}{\z}=0$
and finally
% (A3): 
$\a'\wedge (U_\theta\b'+V_\theta\z)=y\wedge y^\perp$.
These three conditions become
% % (A4): $\a\perp\beta$
% 
% 
% \noindent More precisely, we get 
\begin{align*}
%   \tag{A1} 
  & 
  W_\theta^2q^2(\xi^2+\eta^2)\abs{y}^2
   =(\xi^2+\eta^2)\abs{y}^2 + V_\theta^2 \abs{x}^2
\\& 
% \tag{A2} 
S_\theta(\xi^2+\eta^2)\abs{y}^2+ U_\theta V_\theta \abs{x}^2=0
\\& 
% \tag{A3}
q(\xi^2+\eta^2)=1.
% \\&\tag{A4}
% \a_1\xi+\a_2\eta=0
\end{align*}
We used the fact that $x,y,y^\perp$ are pairwise 
orthogonal and $\abs{y}=\abs{y^\perp}$.
Observe that  $q>0$ and that the  unknown $\xi$ and $\eta$ appear in the form  $(\xi^2+\eta^2)$. Therefore we will find an infinite family of curves.

Eliminating $(\xi^2+\eta^2)$, the first two equations become 
\begin{equation}\label{gattino} 
%  \left\{
%  \begin{aligned}
%   & 
  W_\theta^2 q\abs{y}^2 =
  \frac{\abs{y}^2}{q} +V_\theta^2 \abs{x}^2
  \qquad\text{and}\qquad 
%   \\&
 \frac{S_\theta}{q}\abs{y}^2   +U_\theta V_\theta\abs{x}^2=0.
%  \end{aligned}
% \right.
\end{equation}
Note that $U_\theta>0$ for all $\theta>0$. Moreover,  it cannot be $V_\theta=0$, because in such case it should be $S_\theta=0$ too, but this is impossible because $V$ ans $S$ do not have common zeros on 
$\mathopen]0,+\infty\mathclose[$.
Therefore,  $q =- \dfrac{S_\theta\abs{y}^2} {U_\theta V_\theta \abs{x}^2}$, and since $q>0$, we 
discover that it must be
$
 S(\theta )V(\theta)<0$. 
Eliminating $q$ from the first equation of \eqref{gattino},  we obtain 
\[
 -W_\theta^2\frac{S_\theta\abs{y}^4} {U_\theta V_\theta \abs{x}^2} =
 - \frac{U_\theta V_\theta \abs{x}^2}{S_\theta}+V_\theta^2 \abs{x}^2=-\frac{W_\theta V_\theta }{S_\theta}
 \abs{x}^2,
\]
because $W_\theta=U_\theta-S_\theta V_\theta$. Then we have
\begin{equation*}
 \frac{W(\theta)S(\theta)^2}{U(\theta)V(\theta)^2}= \frac{\abs{x}^4}{\abs{y}^4}.
\end{equation*} 
Taking the square root, and remembering that $U_\theta,W_\theta>0$ for all $\theta>0$, while, as we already observed,   the product   $S_\theta V_\theta $ must be negative,  we conclude that \eqref{equazioniamo} must hold.

  To complete the proof of \ref{itemquattro},   we obtain the explicit form of $\a',\b' $ and $\z$.
 We start  from 
\begin{equation*}  
\xi^2+\eta^2=\frac 1q=-\frac{U_\theta V_\theta }{S_\theta }\frac{\abs{x}^2}{\abs
{y}^2} =(U_\theta W_\theta )^{1/2},
\end{equation*}
where we used  \eqref{equazioniamo} in the last equality. 
Thus   $q=\dfrac{1}{(U_\theta W_\theta)^{1/2}}$,    $\xi=(U_\theta W_\theta )^{1/4}\cos\s$ and $\eta=(U_\theta W_\theta )^{1/4}\sin\s$ for some $\s\in\R$. Therefore, 
% inverting the change of variables \eqref{alexis},  and 
using  \eqref{laddo2} and \eqref{lapponie}, it  turns out that
\begin{equation}
\begin{aligned}
 \a' &=\frac{1}{(U_\theta W_\theta)^{1/4}}(y\sin\s-y^\perp\cos\s)
 \\
 \b' & =\frac{(U_\theta W_\theta)^{1/4}}{W_\theta}(y\cos\s+y^\perp\sin\s) -\frac{V_\theta}{W_\theta}x
 \\
 \z &=\frac{U_\theta}{W_\theta}x -\frac{S_\theta}{W_\theta}(U_\theta 
 W_\theta)^{1/4}(y\cos\s+y^\perp\sin\s).
\end{aligned}
\end{equation} 
Inverting  
the change of variables \eqref{alexis}, 
$
 \a=\a'\sin\theta+\b'\cos\theta
$ and $\beta=\b'\sin\theta-\a'\cos\theta$, we conclude the proof of  
part \ref{itemquattro}.

\step{Proof of \ref{itemuno}}. 
It suffices to check that, 
  given $(x,y\wedge y^\perp)\in\Sigma$,  if $\theta$ satisfies \eqref{equazioniamo},
then for all $\s\in\R$ the triple in \eqref{corresponding} is admissible and   the path $\gamma$ in \eqref{camminassimo} satisfies 
\[
 \gamma(1)=:F(\a'\sin\theta+\b'\cos\theta,\b'\sin\theta-\a'\cos\theta,\z,\theta)=(x,y\wedge y^\perp).
\]
But we have $F(\a'\sin\theta+\b'\cos\theta,\b'\sin\theta-\a'\cos\theta,\z,\theta)=G(\a',\b',\z,\theta)$ by Remark
\ref{cambiovariabile}.  Thus we must check first that $\a',\b',\z$ is admissible and second that
$G(\a',\b',\z,\theta)=(x,y\wedge y^\perp$, i.e.
\begin{equation}\label{arrabatto} 
 S_\theta\b'+\z=x\qquad\text{and}\qquad \a'\wedge (U_\theta \b'+V_\theta \z)=y\wedge y^\perp.
\end{equation} 
 To check admissibility let us  start from a triple as in \eqref{corresponding} and calculate 
 $\abs{\a'}^2  =\dfrac{\abs{y}^2}{(U_\theta W_\theta)^{1/2}}
$
 \begin{equation}\label{alpha} 
\begin{aligned}
 % \\
  \abs{\b'}^2          
&= \frac{(U_\theta W_\theta)^{1/2}}{W_\theta^2}
\abs{y}^2+\frac{V_\theta^2}{W_\theta^2}\abs{x}^2
=\text{(eliminating $\abs{y}^2$ by \eqref{equazioniamo})}
\\&=\Bigl(\frac{-U_\theta V_\theta}{S_\theta W_\theta^2}+\frac{V_\theta^2}{W_\theta^2}\Bigr)\abs{x}^2
=-\frac{V_\theta}{S_\theta W_\theta}\abs{x}^2.
 \end{aligned}
 \end{equation} 
Then $\abs{\a'}^2=\abs{\b'}^2$ if \eqref{equazioniamo} holds.

From \eqref{corresponding} it is obvious 
that $\scalar{\a'}{\z}=0=\scalar{\b'}{\a'}$.
  To  check that $\b'\perp\z$,   it suffices to observe that 
\begin{equation*}
\begin{aligned}
 \scalar{\b'}{\z} & = -\frac{S_\theta (U_\theta W_\theta )^{1/2}}{W_\theta^2}\abs{y}^2 
 -\frac{U_\theta V_\theta }{W_\theta^2 }\abs{x}^2=0,
\end{aligned}
\end{equation*}
as soon as \eqref{equazioniamo} holds. 
Next we check that the $x$-component takes the correct value,    i.e.~the first equality in 
\eqref{arrabatto}. 
\begin{equation*}
\begin{aligned}
 S_\theta \b'+\z & =
 S_\theta \Bigl[\frac{(U_\theta W_\theta)^{1/4}}{W_\theta}(y\cos\s+y^\perp\sin\s)-\frac{V_\theta}{W_\theta}x\Bigr]
\\
&\qquad +
 \Bigl[\frac{U_\theta}{W_\theta}x-\frac{S_\theta}{W_\theta}(U_\theta W_\theta)^{1/4}(y\cos\s+y^\perp\sin\s) \Bigr]
% \\&
 =\Bigl(\frac{-S_\theta V_\theta}{W_\theta}+ \frac{U_\theta}{W_\theta}\Bigr)x
=x,
%  \\&=\frac{U_\theta-s_\theta v_\theta}{W_\theta}(U_\theta W_\theta)^{1/4} (y\cos\s+y^\perp\sin\s)
% =(U_\theta W_\theta)^{1/4} (y\cos\s+y^\perp\sin\s)
\end{aligned}
\end{equation*}
as required. 
To   check the $t$-component   (second equality in \eqref{arrabatto}),  note that,   taking $\b'$ and $\z'$ as in \eqref{corresponding}, we have  
\begin{equation*}
\begin{aligned}
 U_\theta \b'+V_\theta\z & 
%  =
%  s_\theta \Bigl[\frac{(u_\theta W_\theta)^{1/4}}{w}(y\cos\s+y^\perp\sin\s)-\frac{V_\theta}{W_\theta}x\Bigr]
% \\
% &\qquad +
%  \bigl[\frac{U_\theta}{W_\theta}x-\frac{s_\theta}{W_\theta}(U_\theta W_\theta)^{1/4}(y\cos\s+y^\perp\sin\s) \bigr]
% % \\&
% %  =\Bigl(\frac{-s_\theta V_\theta}{W_\theta}+ \frac{U_\theta}{W_\theta}\Bigr)x
% % =x
%  \\&
 =\frac{U_\theta-S_\theta V_\theta}{W_\theta}(U_\theta W_\theta)^{1/4} (y\cos\s+y^\perp\sin\s)
=(U_\theta W_\theta)^{1/4} (y\cos\s+y^\perp\sin\s).
\end{aligned}
\end{equation*}
Then 
\begin{equation*}
 \a'\wedge(U_\theta\b'+V_\theta\z)=
 \frac{1}{(U_\theta W_\theta)^{1/4}}(y\sin\s -y^\perp\cos\s)  \wedge
 (U_\theta W_\theta)^{1/4} (y\cos\s+y^\perp\sin\s)=y\wedge y^\perp,
\end{equation*}
as   desired,  and the proof of \ref{itemuno} is concluded.

\step{Proof of \ref{itemdue}.} From \eqref{alpha} we already know that $ \abs{\b'}^2         =-\dfrac{V_\theta}{S_\theta W_\theta}\abs{x}^2 $.
Moreover, \begin{equation}\label{zetagreco} 
\begin{aligned}
 \abs{\z}^2& =\frac{U_\theta^2}{W_\theta^2}\abs{x}^2+\frac{S_\theta^2}{W_\theta^2}(U_\theta W_\theta)^{1/2}\abs{y}^2
 =\text{(eliminating $\abs{y}^2$ by \eqref{equazioniamo})}
%  \\&
 =\frac{U_\theta}{W_\theta}\abs{x}^2.
\end{aligned}
\end{equation}
Summing up, we find
\begin{equation*}
\begin{aligned}
\length^2(\gamma|_{[0,1]})&= \frac{\abs{x}^2}{W_\theta}\Bigl(U_\theta-\frac{V_\theta}{S_\theta}\Bigr)
=\frac{\abs{x}^2}{W_\theta}\Bigl(W_\theta+S_\theta V_\theta-\frac{V_\theta}{S_\theta}\Bigr)
\\&= \abs{x}^2-\frac{V_\theta(1-S_\theta^2)}{S_\theta W_\theta}\abs{x}^2
=\abs{x}^2+\frac{ 1-S_\theta^2} {\sqrt{U_\theta W_\theta}}\abs{y}^2,
\end{aligned}
\end{equation*}
where the last equality follows again from~\eqref{equazioniamo}.

\step{Proof of \ref{itemtre}.} It follows immediately from \eqref{zetagreco} and \eqref{alpha}.
\end{proof}

\subsection{Minimizing curves and  distance on \texorpdfstring{$\Sigma$}{Sigma}}
Here, among the controls described in Theorem \ref{unmaintheorem}, we will find the minimizing ones.
 Observe that   $\sin(\theta) V(\theta)<0$ if and only if $\theta\in \bigcup_{k=1}^\infty \mathopen]
 k \pi, 
 \phi_k\mathclose[$, where 
 $\phi_k\in\Bigl]k\pi,(k+\frac 12)\pi\Bigr[$ is the $k$-th positive zero of the function $v$.

 Differentiating the functions $S,U,V$ we find the useful formulas 
% \begin{subequations}
% \begin{align}
%  \label{derivouno} s'(\theta)&=-2V(\theta)
%  \\
% \label{derivodue}  u'(\theta)&=\frac{\cos\theta}{\theta}V(\theta)
%  \\
% \label{derivotre}  v'(\theta)&=\frac{s(\theta)}{2}-\frac{2}{\theta}V(\theta)  \text{  se serve...}
% \end{align}
% 
% \end{subequations}
\begin{equation}
 \label{derivouno} S'(\phi)=-2V(\phi),\qquad
 U'(\phi)=\frac{\cos\phi}{\phi}V(\phi),\qquad
  V'(\phi)=\frac{S(\phi)}{2}-\frac{2}{\phi}V(\phi).
\end{equation}

\begin{lemma}\label{erredue} 
 The function 
 \begin{equation*}
  R(\theta):=\frac{1-S(\theta)^2}{\sqrt{U(\theta)W(\theta)}}
 \end{equation*} is strictly increasing on $A:=\bigcup_{k=1}^\infty \mathopen]k \pi, \phi_k\mathclose[$.
\end{lemma}
% \textbf{PLOT}
% \begin{figure} [ht]
%  \label{figura3}
% %  \centerline{\includegraphics{funzione_s.1} } \caption{ File funzione s. Funzione $s$ su $[0,2\pi]$}
% \includegraphics{function_R.1}   \caption{A plot of the function $R$}
% \end{figure}

The function is actually increasing on the whole half-line 
$\theta>0$, but the statement of the lemma is easier to prove and it is enough for our purposes. Observe the limit behaviour  $R(\theta)\sim 4\theta\to +\infty$, as $\theta\to +\infty$.
\begin{proof}
It suffices to prove that the following function is strictly decreasing on $A$.
\begin{equation}
 f=\frac{UW}{(1-S^2)^2}=\frac{U^2}{ (1-S^2)^2} + \frac{U }{  1-S^2 }\cdot \frac{(-SV )}{1-S^2}
 =:F^2+FG.
\end{equation} 
We are omitting the variable $\theta$ in $U,V,W$.
  On  the  functions $F$ and $G$ observe that 
\begin{equation*}
 F(\theta):=\frac{U}{1-S^2}=\frac{\theta-\sin\theta\cos\theta}{4(\theta^2-\sin^2\theta)}=\frac 18
 \frac{d}{d\theta}\log(\theta^2-\sin^2\theta)=
 \frac 18
 \frac{d}{d\theta}\log(\theta^2(1-S^2 ) )
\end{equation*} 
By formula $S'=-2V$, we have
\begin{equation*}
 G (\theta):=\frac{-SV}{1-S^2}=\frac{SS'}{2(1-S^2)}=-\frac14\frac{d}{d\theta}\log(1-S^2).
\end{equation*} 
Thus, 
\begin{equation*}
 G+ 2F=\frac 14 \frac{d}{d\theta}\log(\theta^2)=\frac{1}{2\theta}.
\end{equation*} 
We will show that $f'<0$ on any interval $\mathopen]k\pi, \phi_k\mathclose[$
and that $f(\phi_k)< f((k+1)\pi)$ for all $k\in\N$.
\begin{equation}\label{girodita} 
f'= (F^2+FG)'=F(2F'+G')+ F'G= -\frac{F(\theta)}{2\theta^2}+ F'(\theta)G(\theta).  
\end{equation} 
A computation gives
% $F$ e'  positiva e monotona decrescente su $\R^+$ perche'
\begin{equation*}
\begin{aligned}
 F'(\theta)&=\frac{\sin^2 \theta(\theta^2-\sin^2 \theta)-(\theta-\sin\theta\cos\theta)^2}
 {2(\theta^2-\sin^2\theta)^2}\\
 & =-\frac{(\theta^2 \cos^2\theta+ \sin^2 \theta -2\theta \sin \theta \cos \theta)}{2(\theta^2-\sin^2\theta)^2}
%  \\ & 
=
-\frac{(\theta \cos\theta- \sin \theta)^2}{2(\theta^2-\sin^2\theta)^2}< 0,
\end{aligned}
\end{equation*}
on $  \mathopen]0,+\infty\mathclose[  \setminus\bigcup_{k\geq 1}\{\phi_k\}$. Thus $F$ is strictly decreasing on $\mathopen]0,+\infty\mathclose[  $.
Since  $G$ and $F$ are positive on each interval $\mathopen]k\pi,\phi_k\mathclose[$, 
looking at \eqref{girodita}, we conclude that $f'<0$ on each $\mathopen]k\pi,\phi_k\mathclose[$.

To conclude the proof, observe that $G(k\pi)=G(\phi_k)=0  $ for all $k\in\N$. 
Thus we get   \[
f(\phi_k)=F^2(\phi_k)> F^2 ((k+1)\pi)=f((k+1)\pi)\]
for each $k\geq 1$,
because $F$ is strictly decreasing on $\mathopen]0,+\infty\mathclose[ $.
\end{proof}

\begin{lemma}\label{piopio} 
 The function $P$ defined in \eqref{equazioniamo} is strictly increasing on $\mathopen]\pi,\phi_1\mathclose[$ and $P(\pi)=0$, $P(\phi_1-)=+\infty$.
\end{lemma}

\begin{proof}[Proof of Lemma \ref{piopio}]
  Since $P>0$ on $\mathopen]\pi,\phi_1\mathclose[$,  we shall prove that $P^2$ is increasing on $\mathopen]\pi,\phi_1\mathclose[$.
 \begin{equation*}
  P^2=\frac{WS^2}{UV^2}=\frac{(U-SV)S^2}{UV^2}=\Bigl(-\frac{S}{V}+\frac{S^2}{U}\Bigr)\cdot\Bigl(-\frac{S}
  {V}\Bigr).
 \end{equation*}
 that $V>0$ and  $S<0$ in $\mathopen]\pi,\phi_1\mathclose[ $. Thus  $SV<0$. 

To show the statement it suffices to prove that both the  positive 
functions $ -S/V$ and $S^2/U$ are increasing   
on $\mathopen]\pi,\phi_1\mathclose[ $.
Keeping \eqref{derivouno} into account, we find
\begin{equation*}
\begin{aligned}
 V^2\Bigl(\frac{-S}{V}\Bigr)' & = - S' V + SV'= 2V^2 +\frac{S^2}{2}-\frac{2SV}{\theta}>0,
\end{aligned}
\end{equation*}
because $VS<0$.
Furthermore
\begin{equation*}
 U^2\Bigl(\frac{S^2}{U}\Bigr)' =2SS'U-S^2U'=-4 VSU-S^2\frac{\cos\theta}{\theta}V> 0,
\end{equation*}
because $\cos\theta<0$ in $\mathopen]\pi,\phi_1\mathclose[$. The proof is concluded.
\end{proof}
Next we are ready to prove Theorem~\ref{secondoteorema}.   We will denote by  $P^{-1}:[0,+\infty\mathclose[\to [\pi,\phi_1\mathclose[$  the inverse function of $P\big|_{[\pi,\phi_1\mathclose[}$, which is well defined by Lemma \ref{piopio}.

\begin{proof}
 [Proof of Theorem~\ref{secondoteorema}]
Let $(x,t)\in \Sigma$. Let $\theta\in\mathopen]\pi, \frac 32\pi\mathclose[$
be the smallest solution of $P(\theta)=\frac{\abs{x}^2}{\abs{t}}$. Write $t=y\wedge y^\perp$
for some $y,y^\perp$ orthogonal and of equal length. By Theorem~\ref{unmaintheorem}, item~\ref{itemuno}, we know that for all $\sigma\in\R$ the curve $\gamma_\s$ in 
\eqref{camminassimo} satisfies $\gamma_\s(0)=(0,0)$, $\gamma_\s(1)=(x,t)$. By item \ref{itemdue} of the same theorem, its length is 
\[
 \length^2(\gamma_\s|_{[0,1]}) = \abs{x}^2+\frac{1-S(\theta)^2}{\sqrt{U(\theta)W(\theta)\,}}\abs{t}.
\]
Finally, by Lemma \ref{erredue}, $\gamma_\s$ is length minimizing on $[0,1]$. The proof is  concluded.
\end{proof}

  We are now ready to prove  the following inclusion.  
 \begin{proposition}
  \label{itemcinque}  We have the inclusion
 $\Sigma\subset \operatorname{Cut}_0$.
 \end{proposition}
    The proof of the opposite inclusion will be achieved at the end of Subsection \ref{quattrouno}.

\begin{proof}[Proof of Proposition \ref{itemcinque}]
Assume by contradiction that a point $(x,t)\in \Sigma$ does not belong to $\Cut_0$.   We will show by a standard nonuniqueness argument that we can find a constant-speed 
path  $\Gamma$ which minimizes for $s\in [0,1+\e]$, such that $\Gamma(0)=0, \Gamma(1)=(x,t)$ and $\Gamma$ is not differentiable at $s=1$. This contradicts the known fact that arclength length-minimizers in two step Carnot groups are   normal and then  smooth.

To prove the claim, observe first that, since $(x,t)\in\Sigma$ does not belong to $\Cut_0$,   we can find $y,y^\perp$ and a minimizer $s\mapsto\gamma_{\s_1}(s) $ of the form   \eqref{camminassimo}, i.e. 
\[
 s\mapsto \gamma(s, \a'_{\s_1}\sin\theta +
 \b'_{\s_1}\cos\theta , \b'_{\s_1}\sin\theta-\a'_{\s_1}\cos\theta,\z ,2\theta)=:\gamma(s)=:\gamma_{\s_1}(s)
% \gamma_1(s)=\gamma(s,\a'_\s
\]
such that $\gamma_{\s_1}(1)=(x,t)$ and $\gamma_{\s_1}$ minimizes for $s\in[0,1+\e]$ for some $\e>0$. The minimizing choice of   $\theta \in\mathopen]\pi, \phi_1\mathclose[$ 
is uniquely determined, by Lemma  \ref{erredue} and Lemma  \ref{piopio}.   

Let us consider $\sigma_2\neq \s_1$ and the corresponding path  $\gamma_{\sigma_2}(s)  $  in \eqref{camminassimo}  corresponding to the same   
choice of $\theta$.
 We know that this path minimizes on $[0,1]$ for any choice of $\sigma_2\in\R$. Therefore, the new path $\Gamma$ defined as 
\[
 \Gamma(s)=\begin{cases}
\gamma_{\s_2}(s)&\text{if $s\in[0,1]$}            
\\ \gamma_{\s_1}(s)& \text{if $s\in[1,1+\e]$}
           \end{cases}
\]
is a constant-speed  length minimizer on $[0,1+\e]$.
To prove the nondifferentiability of $\Gamma$, it suffices to show that $\dot\gamma_1(1)\neq\dot\gamma_2(1)$. Since $\gamma_{\s_1}(1)=\gamma_{\s_2}(1)$, this is equivalent to saying that
 $u_{\s_1}(1)\neq u_{\s_2}(1)$ (see \eqref{malafemmina}).
But, letting $\gamma_\s=(x_\s,t_\s)$ and $\dot x_\s=:u_\s$, 
 a computation shows that for $\sigma\in\R$ we have
\[
\begin{aligned}
 u_\s(1)&=\cos(2\theta)(\a'_\s\sin\theta+\b'_\s\cos\theta)
 +\sin(2\theta)
 (\b'_\s\sin\theta-\a'_\s\cos\theta)  +\z_\s
 \\&
 = -\a'_\s\sin\theta+\b'_\s\cos\theta +\z_\s  
 \\&
   = -\frac{1}{ (U_\theta W_\theta)^{1/4} }(y\sin\s-y^\perp\cos\s)\sin\theta +
   \Bigl(\frac{(U_\theta W_\theta)^{1/4}}{W_\theta} (y\cos\s+y^\perp\sin\s)-\frac{V_\theta}{W_\theta}x\Bigr)
   \cos\theta 
\\&\qquad 	
+ \Bigl(\frac{U_\theta}{W_\theta}x -\frac{S_\theta}{W_\theta}(U_\theta W_\theta)^{1/4}  
(y\cos\s+y^\perp\sin\s)\Bigr)
.
\end{aligned}
\]
The proof will be concluded as soon as we  show that $\sigma\mapsto u_\s(1)$ is a nonconstant function.  

To prove the claim, 
write 
\[
\begin{aligned}
u_\s(1)&= \Bigl(\frac{\sin\theta}{(U_\theta W_\theta)^{1/4}}y^\perp +\frac{\cos\theta}{W_\theta}
 (U_\theta W_\theta)^{1/4} y  -\frac{ S_\theta}{W_\theta}(U_\theta W_\theta)^{1/4}y
\Bigr)\cos\sigma
\\&\qquad  +
 \Bigl(  \frac{\cos\theta}{W_\theta}
 (U_\theta W_\theta)^{1/4} y^\perp 
 -
 \frac{\sin\theta}{(U_\theta W_\theta)^{1/4}}y
-\frac{S_\theta }{W_\theta} (U_\theta W_\theta)^{1/4}y^\perp
 \Bigr)\sin\sigma
\\&\qquad -\cos\theta\frac{V_\theta}{W_\theta}x 
  + \frac{U_\theta}{W_\theta}x
\\& 
=: A\cos\sigma + B\sin\sigma + C.
  \end{aligned}
\]
It is easy to see that a function $H(\sigma)= A\cos\sigma + B\sin\sigma + C$ where $A,B,C $ are given vectors is constant if and only if $A=B=0$. Therefore, since $y\perp y^\perp$,
the function $\sigma\mapsto u_\s(1)$ is constant if and only if 
\[
\begin{aligned}
 \frac{(\cos\theta-S(\theta))}{W_\theta}
 (U_\theta W_\theta)^{1/4}=0 \qquad\text{and}\qquad 
 \frac{\sin\theta}{(U_\theta W_\theta)^{1/4}}=0.
\end{aligned}
\]
But this is impossible, because these condition would imply that $\cos\theta=\sin\theta=0$
(recall that $U_\theta>0$ and $W_\theta>0$ for all $\theta>0$).
\end{proof}

\begin{remark} \label{multipli}  Observe that in the argument of the proof   above we  have shown by explicit computation that for 
any point $(x,t)\in \Sigma$, we can connect the origin with $(x,t)$ with multiple minimizing paths with  different tangent vectors at $(x,t)$. Thus, by a standard argument, one can conclude that  the distance from the origin cannot be differentiable at any such $(x,t)$. 
 In Section \ref{sfogliatella}, we will prove some very precise corner-like estimates at cut points. 
\end{remark}

% \subsection{The classes \texorpdfstring{$\mathcal{L}(\gamma)$}{Lg}}

%  \begin{corollary}
%   The function $h_\cut$ has the form
%   \begin{equation*}
%    h_\cut(\mu)= -
%   \end{equation*}
%  \end{corollary}

 \section{Calculation of \texorpdfstring{$h_\cut$}{hcut} and consequences}\label{sec4}
 In the first part of  this  section (Subsection \ref{quattrouno}) we calculate the cut time for any given extremal and we describe the cut locus. In Subsection \ref{quattrocentodue} we describe the profile of subRiemannian spheres.

\subsection{Cut  time  and cut locus}
\label{quattrouno}

 \begin{lemma}\label{quoziente} 
  The function
  \begin{equation*}
  Q(\theta):=-\frac{U(\theta)S(\theta)}{V(\theta)}
 \end{equation*} 
  is positive and  strictly increasing on $\mathopen[\pi,\phi_1\mathclose[ $. Furthermore, $Q(\pi+)=0$ and $Q(\phi_1-)=+\infty$.
 \end{lemma}

 \begin{figure} [ht]
 \label{figura2}
%  \centerline{\includegraphics{funzione_s.1} } \caption{ File funzione s. Funzione $s$ su $[0,2\pi]$}
 \centerline{\includegraphics{function_Q}}  \caption{Plot  of the function $Q$  for $\theta\in\mathopen]0,\phi_1\mathclose[$.}
\end{figure}

The function $Q$ is actually increasing on the whole interval $[0,\phi_1\mathclose[$, but we need (and prove) such property only in the sub-interval $[\pi,\phi_1\mathclose[$. See the plot in Figure~\ref{figura2}.
\begin{proof}[Proof of Lemma \ref{quoziente}]
 Write
 \[
2  Q(\theta)= ( -S(\theta)(\theta-\sin\theta\cos\theta)  )\cdot \frac{1}{\sin\theta-\theta\cos\theta}=: f(\theta)\cdot\frac{1}{g(\theta)}.
 \]
The function $g$ is positive and decreasing on $\mathopen]\pi,\phi_1\mathclose[\subset
\mathopen]\pi,3/2\pi\mathclose[$, because $g'=\theta\sin\theta<0$.
The function $f$ is instead positive and increasing. To check this property,  observe that
\[
 f'(\theta)=-S'(\theta)(\theta-\sin\theta\cos\theta)-S(\theta)(1-\cos^2\theta+\sin^2\theta)>0,
\]
because $-S'(\theta)=2V(\theta)>0$.   This concludes the proof. 
\end{proof}
 We restate Theorem \ref{teo1.3} as follows. 
 \begin{theorem}[Calculation of the cut time]\label{trentadue} 
  Let $Q$ be the function defined in Lemma \ref{quoziente} and in \eqref{RRR}. Let $h_\cut$
be the function in Proposition \ref{ticutto}. Then
 we have 
\begin{equation}\label{ticotto}
 h_\cut(\mu)= Q^{-1}(\mu^2)\qquad \forall\mu\in\mathopen]0,+\infty\mathclose[.
\end{equation} 
% \item
As a consequence, if    $a,b,z$ is 
any admissible triple, $\phi> 0$ and we consider the corresponding path 
 $s\mapsto \gamma(s,a,b,z,2\phi)=:\gamma(s)$,  the path~$\gamma$
 minimizes up to the time 
 \begin{equation}\label{risotto}
  t_\cut (a,b,z,\phi)=\frac{h_\cut(\abs{z}/\abs{a})}{\phi}=\frac{Q^{-1}(\abs{z}^2/\abs{a}^2)}{\phi}
 \end{equation}
and not further. Finally
 \begin{equation}\label{tutte}
  \gamma(t_\cut,a,b,z, 2\phi )\in\Sigma\quad\text{for all admissible $a,b,z$ and $\phi>0$}.
 \end{equation} 
%  \end{enumerate}
\end{theorem}

%  Cancellerei la remark in footnote perch\'e diciamo la stessa cosa e anche meglio nell'introduzione. Dimmi se sei d'accordo\color{black}\footnote{\begin{remark}
%  The cut time has a ``discontinuous'' behaviour. Indeed,   fixed  admissible vectors $a,b,z$ and $\phi>0$,  
% we have \[
%   t_\cut(\e a,\e b, z,\phi)\leq \frac{\phi_1}{\phi},\quad\text{while}\quad 
%     t_\cut(0,0,z,\phi)=+\infty
%  \]
% %Such discontinuity can be seen from the Hamiltonian point of view as follows: if 
% %\[
% %H(x,t,\xi,\tau) :=\frac 12 \sum_{j=1}^3 \scalar{X_j(x,t)}{(\xi,\tau)}^2
% %\]
% %denotes the subRiemannian Hamiltonian and $(\xi,\tau)\mapsto \Exp(\xi,\tau)$ denotes the subRiemannian exponential, then denoting with $T_\cut(\xi,\tau)$ the cut time of the extremal curve with initial covector $(\xi,\tau)$, it turns out that the function $T_\cut$ is smooth on $\{(\xi,\tau):\Exp (\xi,\tau)(\Abn_0)$
% \end{remark}
% }

\begin{remark}\label{quattrocchi} 
 $\Cut_0$ is a four dimensional smooth submanifold of the six-dimensional 
$\F_3$. Indeed it can be parametrized as follows
 \begin{equation*}
  \Cut_0=\Bigl\{ (\lambda x  e_1+\lambda y e_2+\lambda z e_3, z e_1\wedge e_2 -y e_1\wedge e_3+x e_2\wedge e_3 )
  : \lambda\in\R\quad x^2+y^2+z^2>0\Bigr\}
 \end{equation*}
and it is easy to check that the map
$
\Phi(\lambda,x,y,z) =(\lambda x,\lambda y,\lambda z,x,y,z)\in\R^6
$
has full rank on the open set $\{(\lambda,x,y,z): \lambda\in\R\quad x^2 + y^2 + z^2 >0\}\subset \R^4$.  By Theorem~\ref{secondoteorema}, we see that the restriction of the distance from the origin to the cut locus is smooth. 
\end{remark}

\begin{proof}
 [Proof of Theorem \ref{trentadue}]
 We first prove that $h_\cut(\mu)=Q^{-1}(\mu^2) $. Start 
 with an arbitrary $\theta\in\mathopen]\pi,\phi_1\mathclose[$. Choose any 
 $(x,y\wedge y^\perp)\in\Sigma$ such that $P(\theta)=\frac{\abs{x}^2}{\abs{y}^2}$. By Theorem  \ref{unmaintheorem}   and Lemma \ref{erredue},  we may choose (infinitely many) admissible triples
 $\a',\b',\z$  such that the path  $\gamma_\s$ 
 in \eqref{camminassimo}   satisfies $\gamma_\s(1)=(x,y\wedge y^\perp)$, 
 minimizes length up to time $t_\cut=1$ and not further  (we already know from Proposition \ref{itemcinque} that $\Sigma\subseteq \Cut_0$). 
 Thus,
 \begin{equation*}
\begin{aligned}
  1 & =t_\cut(
  \a'\sin\theta +\b'\cos\theta , \b'\sin\theta-\a'\cos\theta,\z,\theta)
%   \\&=
=\frac{  h_\cut(\abs{\z}/\abs{\a'})}{\theta},
\end{aligned}
 \end{equation*}
where last equality follows from Proposition~\ref{ticutto}.
Then $ h_\cut(\abs{\z}/\abs{\a'})=\theta$. By item \ref{itemtre} of Theorem \ref{unmaintheorem} we get
$
 h_\cut(\sqrt{Q(\theta)})=\theta
$. Since $\theta\in \mathopen]\pi,\phi_1\mathclose[$ is arbitrary, Lemma \ref{quoziente} gives the conclusion $h_\cut(\mu)=Q^{-1}(\mu^2)$ for all $\mu\in\mathopen]0,+\infty\mathclose[$.
Thus we have proved \eqref{ticotto} and \eqref{risotto}.

Next we prove \eqref{tutte}. Let $a,b,z$ be any admissible triple and let $\phi>0$. 
Recall that the admissibility condition ensures that   $\abs{a}=\abs{b}> 0$. 
By the first part of the theorem, we know that 
 \begin{equation*} 
    t_\cut (a,b,z,\phi)=\frac{h_\cut(\abs{z}/\abs{a})}{\phi} 
    = \frac{Q^{-1}(\abs{z}^2/\abs{a}^2)}{\phi}.
 \end{equation*} 
  To prove \eqref{tutte},  we claim that,   under the choice
\begin{equation}
\theta : = Q^{-1}(\abs{z}^2/\abs{a}^2), \label{trecinque}
\end{equation}  
(i.e.~$t_\cut(a,b,z,\phi)=:\frac{\theta}{\phi}$) we have $
  \gamma\Bigl(\frac{\theta}{\phi}, a,b,z,2\phi\Bigr)\in\Sigma$. 

  To prove the claim, note that  the invariance property \eqref{riparametrizzo} gives
\begin{equation*}
\begin{aligned}
   \gamma\Bigl(\frac{\theta}{\phi}, a,b,z,2\phi\Bigr) &=
   \gamma\bigg(1, \frac{\theta}{\phi}a,\frac{\theta}{\phi}b,\frac{\theta}{\phi}z,2\theta\bigg)
  =F\bigg(\frac{\theta}{\phi}a,\frac{\theta}{\phi}b,\frac{\theta}{\phi}z,\theta\bigg),
\end{aligned}
\end{equation*}
by the definition of $F$.
Since the set $\Sigma$ is dilation invariant, the latter point belongs to $\Sigma$ if and only if its $\frac{\phi}{\theta}$-dilated 
$F(a,b,z,\theta)$ belongs to $\Sigma$. Here we need the dilation invariance property of 
$F$. \footnote{That is, the property $F(r\a,r\b,r\z,\phi)=\delta_r F(\a,\b,\z,	\phi)$ for all $r>0$, $\a,\b,\z$ admissible and $\phi>0$. }
But 
\begin{equation*}
\begin{aligned}
F(a,b,z,\theta)&= 
\Bigl(S_\theta (a\cos\theta+b\sin\theta)+z\,,\;
 U_\theta a\wedge b+ V_\theta (a\sin\theta -b\cos\theta)
 \wedge z\Bigr)
%  \\&=
% \Bigl(s_\theta (a\cos\theta+b\sin\theta)+z\,,\;
% (a\sin\theta -b\cos\theta)
% \wedge \bigl\{
% U_\theta (a\cos\theta+b\sin\theta)+V_\theta z\bigr\}\Bigr)
\\&
=:\Bigl(S_\theta b' +z\,,\;
a'
\wedge \bigl\{
U_\theta b' +V_\theta z\bigr\}\Bigr),
\end{aligned}
\end{equation*}
where $a',b'$ are defined \eqref{cambiovaluta} and we keep Remark~\ref{cambiovariabile} into account.
 Since $a',b',z$ is an admissible triple, we have for free that $S_\theta b'+z\perp a'$, and ultimately the point $F(a,b,z,\theta)$ belongs to $\Sigma$ if and only if $S_\theta b' +z\perp U_\theta b' +V_\theta z$, i.e. if 
\begin{equation*}
 U_\theta S_\theta \abs{b}^2+V_\theta\abs{z}^2=0, 
\end{equation*}
which holds true because our choice of $\theta$ in \eqref{trecinque} 
\end{proof}

Theorem \ref{trentadue} shows that, if $a,b,z$ is admissible and $\phi>0$, then 
%  (Frase da mettere prima dell'enunciato del teorema?)
\begin{equation}\label{minimizzando} 
 d(0,\gamma(s,a,b,z,2\phi))=s\sqrt{\abs{a}^2+\abs{z}^2}\quad\text{ if  }\quad 0\leq  s\leq \frac{1}{\phi}h_\cut
 \Bigl(\frac{\abs{z}}{\abs{a}}\Bigr).
\end{equation} 
% and proves that the cut locus agrees with the set $\Sigma$.

 Finally  we are in a position to prove Theorem \ref{teo1.1}. 
\begin{proof}[Proof of Theorem \ref{teo1.1}]
Theorem \ref{trentadue}  and in particular formula  \eqref{tutte} imply easily the inclusion
$\Cut_0\subset\Sigma$. The opposite inclusion has been proved in Proposition \ref{itemcinque}. 
\end{proof}

\subsection{Profile of the unit sphere}\label{quattrocentodue}
  Let for all admissible $\a,\b,\z$  and for any $\theta>0$ 
\begin{equation*}
 G(\a,\b,\z,\theta):=\Bigl( S(\theta)\b+\z\;,\;\a\wedge(U(\theta)\b + V (\theta)\z)\Bigr).
\end{equation*} 
The profile of the subRiemannian sphere can be described as follows
\begin{corollary}[Profile of the   subRiemannian  sphere]\label{proffilo} 
  Let $S(0,r):=\{(x,t)\in\F_3 :d(x,t)=r\}$  be the subRiemannian sphere of radius $r$. Then 
 \begin{equation*}
\begin{aligned}
S(0,r):& 
%  =\{(x,t): d((0,0),(x,t))=1\}
% \\&=
=\Big\{ G (\a,\b,\z,\theta): \a,\b,\z \text{ admissible triple  with } \abs{\a}^2+\abs{\z}^2=  r^2
\\&\qquad \qquad \qquad \quad\text{and}\quad 0\leq\theta\leq
h_\cut\bigl( \abs{\z}/\abs{\a } \bigr)    \Big\}.     
\end{aligned}
\end{equation*}
\end{corollary}
\begin{proof}
  By dilation invariance we may assume $r=1$.  Start from \eqref{minimizzando} and observe that
 \[
\begin{aligned}
  S(0,1)& =\bigl\{\gamma(s,a,b,z,2\phi)\mid s^2(\abs{a}^2+\abs{z}^2)=1 
  \quad\text{and}\quad 0<s\leq  \phi^{-1}
  h_\cut \bigl( \abs{z}/ \abs{a}\bigr) \bigr\}.
 \\ &
=\bigl\{F( sa,sb,sz,s\phi)\mid s^2(\abs{a}^2+\abs{z}^2)=1 
  \quad\text{and}\quad 0<s\leq  \phi^{-1}
  h_\cut \bigl( \abs{z}/ \abs{a}\bigr) \bigr\}.
 \\ &
=\bigl\{F( \a,\b,\z,\theta)\mid  \abs{\a}^2+\abs{\z}^2=1 
  \quad\text{and}\quad 0\leq \theta \leq   
  h_\cut \bigl( \abs{\z}/ \abs{\a}\bigr) \bigr\}
  \end{aligned}
 \]
and the required thesis follows. In the last chain of  equalities, $a,b,z$ and $\a,\b,\zeta$ are always admissible triples.
 % Observe that
% \[
%   \gamma(s,a,b,z,2\phi)=\gamma(1,sa,sz,2s\phi)=F(sa,sb,sz,2s\phi).
% \]
% Thus,  
\end{proof}

%\end{document}

\section{Some remarks on the Hamiltonian point of view}\label{sec5}
 In this section we consider the Hamiltonian point of view. In particular we will calculate the cut time as an explicit function of the initial covector.  
In order to look   at  the Hamiltonian point of view, it is convenient to represent the group $\F_3$ in the form $\R^3_x\times \R^3_t$ with the group law
\begin{equation}\label{crossiamo} 
 (x,t)\circ (x',t')=\Bigl(x+x', t+t'+\frac 12 x\times x'\Bigr),
\end{equation}
where $\times $ denotes the standard cross product in $\R^3$. This law is not scalable to higher dimensional cases, but it is rather 
convenient in the rank three case (compare \cite{MartiniMuller13}).
  
The standard basis of horizontal vector fields is $Y_j=\Bigl(e_j,\frac 12 x\times e_j\Bigr)\in\R^3\times \R^3$, for $j=1,2,3$. Namely
\[
 Y_1=\p_{x_1}+\frac 12 x_3 \p_{t_2}-\frac 12 x_2\p_{t_3},
 \quad
 Y_2=\p_{x_2} -\frac 12 x_3\p_{t_1}+\frac 12 x_1\p_{t_3},
 \quad 
 Y_3 = \p_{x_3}+\frac 12 x_2\p_{t_1}-\frac 12 x_1\p_{t_2}.
\]
Commutation relations
take the form $
 [Y_1, Y_2]=\p_{t_3},$ $[Y_1, Y_3]= -\p_{t_2},$  and $[Y_2,  Y_3]=\p_{t_1}$. 
If we  denote by $d$ the distance generated by $Y_1,Y_2, Y_3$ and by $\Cut_0$ the cut locus of the origin of the distance $d$, then it turns out that
\[
 \Cut_0=\{(x,t)\in\R^3\times\R^3: t\neq 0\quad\text{ and $x=\lambda t $ for some $\lambda\in\R$}\}.
\]
   Observe also the invariance property $d(Mx, My\times Mz)=d(x, y\times z)$ for all $x,y,z\in\R^3$ and $M\in O(3)$. 
  
 Denote by $q=(x,t)\in\R^6$ and $p=(\xi,\tau)\in \R^6  $ coordinates in $T^*\R^6$ and define the subRiemannian Hamiltonian 
 \begin{equation}\label{camillo} 
  H(q,p):=
%   \frac 12 \abs{p}_{T^*}^2:=
  \frac 12\sum_{j=1}^3 \scalar{p}{Y_j(q)}^2=
  \frac 12\sum_j\scalar{(\xi,\tau)}{Y_j(x,t)}^2=:\frac 12\sum_j u_j^2(x,t,\xi,\tau).
 \end{equation} 
 where  $\scalar{\cdot}{\cdot}$ denotes the Euclidean inner product in $\R^6$. 
 
Recall the following definition of subRiemannian exponential (see \cite{AgrachevBarilariBoscain}. Let $p\in T^*_{(0.0)}$. 
 Denote by $s\mapsto(q(s,p_0), p(s,p_0))
 $ the solution of the Hamiltonian system associated to 
 \eqref{camillo} with initial data $q(0)=(0,0)$ and $p(0)=p_0=(\xi_0,\t_0)$. Define
 \begin{equation}
  \cal{E}(p_0)= \E(\xi_0,\tau_0)=q(1,p_0).
 \end{equation} 
 Note that in our case solutions are defined globally in time for all initial datum $p_0$. Furthermore, by 
 the  property 
 $(q(s,\a p_0),p(s,\a p_0))=(q(\a s,p_0) ,\a p(\a s, p_0)$, we have $\E(sp)=q(s,p)$ for all $p,s$.
 The curve $s\mapsto q(s,p)=q(s,(\xi,\tau))$ is minimizing on some interval  $[0,  T_\cut(\xi,\tau)]$, where
 \begin{equation}
  T_\cut(\xi,\tau) =\text{the cut time of the path $s\mapsto q(s,(\xi,\tau))$.}
 \end{equation}

As usual in this setting,  instead of $(x_j ,t_{j },\xi_j,\t_{j})$ we may use coordinates  $(x_j,t_{j },u_j,\t_{j })$, where  \begin{equation*}
u_j(x,t,\xi,\tau)=\scalar{Y_j(x,t)}{(\xi,\tau)}
=\xi_j +\frac 12 \scalar{\t}{x\times e_j } 
=\xi_j+\sum_{\s\in S_3}\sigma(j,k,\ell) \t_k x_\ell,
\end{equation*}
where $\sigma(jk\ell)$ denotes the sign of the permutation. 
Hamilton's  equation, see \cite{BarilariBoscainGauthier12}, take the form
\[
\left\{\begin{aligned}
 &\dot u_j =\sum_k\scalar{(\xi,\tau)}{[Y_k,Y_j](x,t)}u_k
 =(\t\times u)_j                          
 \\&
 \dot\t_{j}=0
\\&
\dot x=u
% \sum_j u_jY_j(x,t)
\\&
\dot t=\frac 12x\times u
 \end{aligned}
\right.
\]
with initial condition $x(0)=0$, $t(0)=0$, $u(0)=\xi(0)=u$ and $ \t(0)=\t $.  
% \begin{equation*}
%  \begin{bmatrix}
% \dot u_1 \\ \dot u_2 \\ \dot u_3  
%  \end{bmatrix}
%  =
%  \begin{bmatrix}
% 0 &-\t_{12}&-\t_{13}
% \\
% \t_{12} & 0 & -\t_{23}
% \\
% \t_{13} & \t_{23} & 0
%  \end{bmatrix}
%  \begin{bmatrix}
%  u_1 \\  u_2 \\  u_3  
% \end{bmatrix}
% \end{equation*}

% \begin{proof}[Prima dimostrazione]Seguire Roberto articolo Lincei Calcolare Jacobiano di $G$. Attenzione alla iniettivit\`a??? Serve parametrizzare la famiglia
% $\a,\b,\z 
% $ ammissibili con esattamente cinque variabili (cio\`e il numero di gradi di libert\`a della famiglia delle terne ammissibili). Assieme alla ulteriore $\phi$ arriviamo a esattamente 6 gradi di libert\`a, cio]`e la dimensione topologica di $\F_3\sim\R^6$.
% 
% Per ora non so farlo.
%  Anna hai ifdea di come fare questa parametrizzazione???\color{black} 
% 
% 
% 
% \end{proof}
By the Rodrigues' rotation formula we get that if we consider the skew-symmetric map
\[
% T_\t x: = =%
 \R^3\ni x\mapsto \t\times x= 
  \begin{bmatrix}
    0& -\t_3 & \t_2
    \\
  \t_3 & 0 & -\t_1
  \\
-\t_2 & \t_1 & 0  
                    \end{bmatrix}
\begin{bmatrix}
 x_1 \\ x_2\\ x_3
\end{bmatrix}   =:A_\t x 
                    \]
then we have 
\begin{equation}\label{delafu} 
e^{  sA_\t }u= a\cos(\la s)+b\sin(\la s) + z,                                       \end{equation} 
where 
\begin{equation}\begin{aligned}
\lambda=  \abs{\t},\qquad a=-\frac{\t}{\abs{\t}}\times \Bigl(\frac{\t}{\abs{\t}}\times u\Bigr)
=u-\Scalar{u}{\frac{\t}{\abs{\t}}}\frac{\t}{\abs{\t}}
\\
b=\frac{\t}{\abs{\t}}\times u,\qquad z=\Scalar{u}{\frac{\t}{\abs{\t}}}
\frac{\t}{\abs{\t}}.
% \frac{m}{\abs{m}}\times u,\qquad b= 
 \end{aligned}
\end{equation}               
(These relations can be obtained easily   with a purely analitic argument   comparing derivatives of order $0,1,2$ at $s=0$ of the two sides of \eqref{delafu}).

 Thus we have proved the following facts. 
% Given $\t=(\t_{12}, \t_{13}, \t_{23})$, let $m(\tau)=(\t_{23}, -\t_{13},\t_{12})$.
\begin{proposition} 
 Let $s\mapsto \E (s(\xi,\tau)) =\E(s\xi,s\tau)$ be the   (projection of the)  solution of the Hamiltonian   associated with \eqref{camillo}  with 
 $(x(0),t(0))=(0,0)$ and $(\xi(0),\t(0))=(\xi,\tau)=(\xi_1,\xi_2,\xi_3,\t_{1}, \t_{2}, \t_{3})$.  
  Then the corresponding control $u(s)$ has the form
 \begin{equation}
 \begin{aligned}
 u(s)& =
 \left(\xi-\Scalar{\xi}{\frac{\t}{\abs{\t}}}\frac{\t}{\abs{\t}}\right)\cos(\abs{ \t }s)
 +\left(\frac{\t}{\abs{\t}}\times \xi\right)\sin(\abs{ \t    } s) 
%  \\& \qquad 
 +\Scalar{\xi}{\frac{\t}{\abs{\t}}}\frac{\t}{\abs{\t}}.
 \end{aligned}
 \end{equation}
 In particular
 \begin{equation}
  T_\cut(\xi,\tau) = \frac{2}{\abs{\t}}
   h_{\cut}\Big(\frac{\abs{\scalar{\xi}{\t}}}{\abs{\t\times\xi}}\Big),
 \end{equation} 
 and the function $T_\cut$ is $C^\infty$ smooth on  \begin{equation}
\Omega:=\{(\xi,\tau): \t\times \xi\neq 0\}.                                                                                          \end{equation}  
\end{proposition}

\begin{remark}
  The function  $T_\cut$ is singular if $\xi $ is parallel to $\t$. Namely,  $T_\cut (\xi ,\t)\in \mathopen[\frac{2\pi}{\abs{\tau}},\frac{2\phi_1}{\abs{\tau}}\mathclose[$   for all $(\xi,\tau)\in\Omega $,   while  
  $T_\cut(\mu\t,\t)=+\infty$ for all $\t\neq 0$ and $\mu\in\R\setminus{\{0\}}$. The corresponding final  points  $\mathcal{E}(\mu\t,\t)=(\mu \t,0)$ belong to the abnormal set $   
 \Abn_0$.
\end{remark}
% A complete calculation with a more detailed study of the related exponential map will be performed in a subsequent paper.\footnote{perche' allora mettiamo questi conti qui? qui siamo interessati a far vedere cosa?}

%  \color{blue} %\color{ForestGreen}\color{Purple}

\section{Corner-like estimates at cut points}\label{sfogliatella} 
In this section we show that at any point of the cut locus we can construct a two dimensional $C^1$-smooth surface such that on such surface the distance has a corner-like singularity. Our estimates give  also an affirmative answer in our model to a question raised by Figalli and Rifford \cite{FigalliRifford10} on whether the distance fails to be semiconvex at the cut locus.  

To show our construction, we work in the model used in the previous section, with the group law \eqref{crossiamo} and we   consider a point $(\ol x,\ol t)\in\Cut_0$. Then we may write 
\begin{equation*}
 (\ol x, \ol t)=(S_\phi\b+\z,\a\times ( U_\phi \b+V_\phi \z)),
\end{equation*}
where the triple $\a,\b,\z$ is admissible and
belonging to the cut locus means that $\phi\in[\pi,\phi_1\mathclose[$ satisfies
 $-\frac{U_\phi S_\phi}{V_\phi}=\frac{\abs{\z}^2}{\abs{\b}^2}$. In other words,  $S_\phi\b+\z$ and $U_\phi \b+V_\phi \z$ are perpendicular, which is equivalent to the fact that $\ol x$ and $\ol t$ are parallel. We may also write $\ol t=\ol y\times \ol y^\perp$, where $\ol y,\ol y^\perp\in \Span\{\ol x\}^\perp$ are orthogonal and have the same length. Finally, recall  the formula $d(\ol x,\ol t )^2=\abs{\b}^2+\abs{\z}^2=\abs{\a}^2+\abs{\z}^2$.

Let $\s_0>0$ and let $x$ and $t:\mathopen]-\s_0,\s_0\mathclose[\to \R^3$ 
be smooth functions of the form
\begin{equation}\label{ristorante} 
 (x(\s),t(\s))= \Bigl(\ol x+\s(\mu_1y+\mu_2 y^\perp)+\s^2 r(\sigma), 
 \ol t+\s(\nu_1y+\nu_2 y^\perp)+\s^2 \rho(\sigma)\Bigr)
\end{equation}  
where $r$ and $\rho:\mathopen]-\s_0, \s_0\mathclose[\to \R^3$ are $C^\infty$ vector-valued functions. We always assume that   $(\mu_1,\mu_2)\neq(0,0)$. The vector $(\nu_1,\nu_2)$ can possibly vanish.
% Let $Heisenb\mu_0,\mu_1,\mu_2, \nu_0,\nu_1,\nu_2$ be real valued functions of class $C^1$ on some interval containing the origin and assume that 
% $
% (\mu_1(0),\mu_2(0))\neq 0. $ Define the curves
% \begin{equation*}\begin{aligned}
%  \sigma\mapsto (x(\sigma), t(\sigma))
%  := 
%  \\
%  \Bigl((1+\sigma^2 \mu_0(\s))&\ol x+  \sigma  \{\mu_1(\sigma)y+\mu_2(\sigma)y^\perp\},
% %  \\&\qquad\qquad \qquad 
% (1+\sigma^2 \nu_0(\s)) \ol t+ \sigma  \{\nu_1(\sigma)y+\nu_2(\sigma)y^\perp\}
%  \Bigr) 
%  .    \end{aligned}
% \end{equation*}
% SCRITTURA MENO MACCHINOSA 
% where $r$  and $ \rho :\mathopen]-\s_0,\s_0\mathclose[\to \R^3  $ are smooth functions in 
% Oppure
%  scrittura pi\'u snella
% \[
%  (x(\sigma), t(\sigma)= (\ol x + \sigma (\mu_1 y+\mu_2 y)+\sigma^2\Omega_1(\sigma), 
%  \ol t + \sigma^2\Omega_2(\sigma))
% \]
% 
Let $R(\theta)$ be the rotation of an angle $\theta$ which leaves fixed $\ol x$ 
(we can fix the sign by requiring  for example that for any  vector $\xi \perp\ol x$, the vector
$R(\theta) \xi\times (R'(\theta)\xi)$ points in the same direction of $\ol x$). Then define the parametrization
\begin{equation}\begin{aligned}
 H(\sigma,\theta)& :
  =(R(\theta) x(\sigma), R(\theta) t(\sigma))
\\
=\Bigl( \ol x + & \sigma R(\theta)\{\mu_1 y+\mu_2 y^\perp\}+ \s^2 R(\theta) r(\sigma)
%  \\&\qquad\qquad \qquad 
,  \ol t+ \sigma  R(\theta)\{\nu_1 y+\nu_2 y^\perp\}+\s^2 R(\theta) \r(\sigma)
 \Bigr).
\end{aligned}
\end{equation} 
Since $ (\mu_1,\mu_2)\neq (0,0)$, for small $\sigma_0>0$, the set 
% \begin{equation}
$ \Gamma =\{ H(\sigma,\theta):0\leq \sigma<\sigma_0,$ and $ 0\leq \theta <2\pi\}$
% \end{equation} 
is a $C^1 $-smooth two dimensional  surface containing $(\ol x,\ol t)$. The expansion~\eqref{ristorante} shows  that 
$T_{(\ol x,\ol t)}\Gamma=\Span\{(\mu_1y+\mu_2y^\perp, \nu_1 y+\nu_2 y^\perp),(-\mu_2 y+\mu_1y^\perp, -\nu_2 y+\nu_1 y^\perp) \}$. Furthermore the distance from the origin of points on $\Gamma$ enjoys the rotational invariance  property $d(H(\sigma,\theta_1))=d(H(\sigma,\theta_2))$ for
all $\theta_1, \theta_2$ and $\s$.

We will show that 
\begin{theorem}[Corner-like estimate for the distance at cut points]\label{corner}
 For any point $(\ol x,\ol t)=(S_\phi\beta+\zeta, \a\times(U_\phi\beta+ V_\phi \z))\in \Cut_0 $ there is 
 $\s_0>0$, there are  smooth functions $x  , t $ as in~\eqref{ristorante} with $(\mu_1 ,\mu_2 )\neq (0,0)$ and there is $C>0$ such that  
 \begin{equation}  
  d(H(\sigma,\theta))\leq d(\ol x,\ol t)- C\sigma,\quad\text{ for all $\sigma\in[0,\sigma_0\mathclose[$
  and $\theta\in[0,2\pi\mathclose[$.}
 \end{equation} 
\end{theorem}
This gives a   generalization of the well known estimates at cut points $(0,t)\neq(0 ,0)\in\C\times\R$
of the Heisenberg group $\H^1,$   namely
$
 d(z,t)\leq d(0 , t)-C\abs{z},$  for   $\abs{z}$ small.
One can  recover a similar estimate as a limiting case of Theorem~\ref{corner}, as $\phi=\pi$. 
See Remark~\ref{Heise}.

As a corollary we get an answer in this setting to the question raised by Figalli and Rifford (see the \emph{Open problem} at p~145-146 in  \cite{FigalliRifford10}).
\begin{corollary}[Failure of semiconvexity at cut points]\label{cuttolo} 
 At any cut point $(\ol x, \ol t)\in \Cut_0$, for any neighborhood $\Omega$ of $(\ol x,\ol t)$, we have 
 \begin{equation}
  \inf_{p,q,(p+q)/2\in\Omega}\frac{d(p)+d(q)-2d((p+q)/2)}{\abs{p-q}^2}=-\infty .
 \end{equation} 
 \end{corollary}
 Strictly speaking,  
 to get a complete analogous of the Riemannian estimate (2.6) in Proposition~2.5 
   of \cite{CorderoMcCannSchmuckenschläger01}, our infimum should be calculated with $\frac{p+q}{2}=(\ol x,\ol t)$. Corollary~\ref{cuttolo} will be proved at the end of the section.

 We also remark that in \cite{FigalliRifford10} the authors define the cut locus $\Cut^{\textup{FR}}_0$ as the closure of the 
 set of points $(x,t)\neq (0,0)$ where the  distance is not continuously differentiable. This set is strictly larger than our set $\Cut_0$. Namely, we have  
$\Cut^{\textup{FR}}_0=  \Cut_0 \cup \Abn_0 $, where the union is disjoint and $\Abn_0$ is the set of final points of length minimizing abnormal curves and it is known that 
  $\Abn_0=\{(x,0)\in\R^3\times \R^3\}$ (see \cite{LeDonneMontgomeryOttazziPansuVittone,MontanariMorbidelli15}). At points $(x,0)$, in \cite{MontanariMorbidelli15} we proved that
  \[
   d(x,\s x)\geq d(x,0)+C\abs{\s}, \quad\text{for $\sigma$ close to $0$.}
  \]
Roughly speaking, the presence of abnormals gives rise to the existence of directions where the Hessian of the distance is $+\infty$, i.e.~semiconcavity fails. 
On the contrary, we  are not aware of
the existence of directions where the Hessian is $-\infty$ at abnormal points $(x,0)$.

\begin{proof}[Proof of Theorem \ref{corner}] Let $(\ol x, \ol t)=(S_\phi\b+\z,\a\times (U_\phi \b+V_\phi \z ))\in \Cut_0$.
  We first construct  a smooth  curve $\sigma \mapsto (x(\sigma), t(\sigma))$ for $\sigma   $ in a neighborhood of the origin defined as
\[ (x(\sigma), t(\sigma))=\Bigl( S_{  \phi -\sigma}(1- c_1\sigma)\b + (1-c_2\s)\z,
(1-c_1\s)\a\times \{ U_{\phi-\s}(1-c_1\s)\b+V_{ \phi -\s}(1-c_2\s)\z \} \Bigr).
\]
Observe that, given $\phi\in[\pi,\phi_1\mathclose[ $, the triple $\a(\s)=(1-c_1\s)\a$, $\b(\s)=(1-c_1\s)\b$ and $\z(\s)=(1-c_2\s)\z$ is admissible for all $\sigma$ close to $0$.
Therefore we have the upper estimate
$d(x(\s),t(\s))^2\leq (1-c_1\s)^2\abs{\a}^2+(1-c_2\s)^2\abs{\z}^2$ (which,  by our previous results,  becomes an equality if and only if $-\frac{U_{\phi-\s} S_{\phi-\s}}{V_{\phi-\s}}\leq  \frac{(1-c_2\s)^2\abs{\z}^2}{(1-c_1\s)^2\abs{\a}^2}$).

\medskip \noindent \textit{Step 1.} 
We show that there is a (actually unique) choice of $c_1,c_2$ such that
$x'(0)$ and $t'(0)\in\Span\{y,y^\perp\}$. We start with $x'(0)$. A calculation shows that
\begin{equation*}
 \begin{aligned}
x'(0)
&=\frac{d}{d\s}\Big|_{\sigma =0} S_{ \ol \phi -\s} (1-c_1\s)\b+(1-c_2\s)\z 
% \\&
   = - S'_\phi\b - c_1 S_\phi\b-c_2\z
   \\&
   =(2V_\phi -c_1 S_\phi)\b - c_2\z,
   \end{aligned}
\end{equation*}
by formula $S'_\phi=-2V_\phi$, where we write $S'_\phi=\frac{dS}{d\phi}$. 
We must require that $\scalar{x'(0)}{\ol x}=0$, where $\ol x=S_\phi\b+ \z$. 
Thus
\begin{equation*}
\begin{aligned}
 0=\scalar{x'(0)}{\ol x} &=
%  \Big\langle
 S_\phi(2V_\phi-c_1S_\phi)\abs{\b}^2-c_2\abs{\z}^2 
%  \Big\rangle
 =\frac{S_\phi}{V_\phi}\abs{\b}^2(2V_\phi^2-c_1 V_\phi S_\phi+c_2 U_\phi),
\end{aligned}
\end{equation*}
where  we used the identity $\frac{\abs{\z}^2}{\abs{\b}^2}= -\frac{U_\phi S_\phi}{V_\phi}$ which holds on the cut locus. 
% Then we get 
Thus we have found a first linear equation in $c_1, c_2$.

Next we calculate $t'$.
\begin{equation*}
\begin{aligned}
 t'(0)&= -c_1\a\times (U_\phi \b+ V_\phi \z)+\a\times[-(U'_\phi \b+c_1U_\phi \b
 +V'_\phi \z+V_\phi c_2\z)]
 \\&
   = -\a\times [(2c_1 U_\phi + U'_\phi )\b +((c_1+c_2)V_\phi +V'_\phi )\z].
\end{aligned}
\end{equation*}
Since $x=S_\phi \b +\z$ is orthogonal to $U_\phi \b + V_\phi \z$ on the cut locus, condition $\scalar{t'(0)}{\ol x}=0$ is equivalent to $ \Scalar{
 (2c_1 U_\phi + U'_\phi )\b +((c_1+c_2)V_\phi +V'_\phi )\z}{ U_\phi \b + V_\phi \z}=0.$ 
 A calculation of the inner product gives
\begin{equation*}
\begin{aligned}
&(2c_1 U_\phi + U'_\phi )U_\phi \abs{\b}^2 +((c_1+c_2)V_\phi +V'_\phi )V_\phi \abs{\z}^2=0\quad\text{ i.e.}
\\&
% \Leftrightarrow\quad 
(2U_\phi -S_\phi V_\phi )c_1 - S_\phi V_\phi c_2= S_\phi V'_\phi -U'_\phi ,
\end{aligned}
\end{equation*}
which is again a linear equation in $c_1, c_2$.
Ultimately we have the linear system
\begin{equation*}
\left\{\begin{aligned}
 &  V_\phi S_\phi c_1- U_\phi c_2 =2V_\phi^2 
 \\&
 (2U_\phi -S_\phi V_\phi )c_1 - S_\phi V_\phi c_2= S_\phi V'_\phi -U'_\phi.
\end{aligned}\right.
\end{equation*} 
The solution has the form
\begin{equation}\label{ciunodue} 
 c_1=c_1(\phi) =\frac{-2 S_\phi V^3_\phi - U_\phi U'_\phi +U_\phi S_\phi V'_\phi }
 {2U_\phi ^2 - U_\phi  S_\phi V_\phi -V^2_\phi S^2_\phi }\quad\text{and}
 \quad c_2=c_2(\phi)=\frac{V_\phi }{U_\phi }(S_\phi c_1   -2V_\phi ).
\end{equation} 
 Note that $2U_\phi ^2 - U_\phi  S_\phi V_\phi -V^2_\phi S^2_\phi>0$ for all   $\phi\in [\pi, \phi_1[$,   because $2U_\phi^2>0$, $S_\phi V_\phi\leq 0$ and $U_\phi+S_\phi V_\phi>0$, as it can be easily seen using the  very definition of $U,V,S$.

 Observe  that, if we write $x'(0)$ and $t'(0)$ in terms of $c_1$ only,  it turns out that 
 \[
  x'(0)=\frac{2V_\phi - c_1 S_\phi}{U_\phi}(U_\phi \b + V_\phi \z)\quad\text{ and } 
  \quad 
  t'(0)=-\frac{2c_1 U_\phi + U_\phi'}{S_\phi}\,\a\times (S_\phi \b+\z),
 \]
and ultimately, since $U_\phi \b+ V_\phi \z$ and   $S_\phi \b+\z$   are orthogonal, 
$x'(0)$  and $t'(0)$ are parallel.

%  \medskip\noindent\textit{Step 2.} We have $x'(0)\neq 0$. Since $x'(0$
%  Observe that $c_1>0$ (indeed, by formulas for the derivatives 
%  This concludes the proof of Step 1. 
% 

\medskip\noindent\textit{Step 2.} Concerning the curve constructed above, we show that there is $C>0$ and $\s_0>0$ such that 
\begin{equation*}
 d(x(\s), t(\s))\leq d(\ol x,\ol t)- C\s \quad\text{for all $\sigma\in[0,\s_0\mathclose[$. }
\end{equation*}
To prove such estimate, start from the upper estimate
\begin{equation*}
\begin{aligned}
d(x(\sigma), t(\sigma))^2 & 
% \leq \length(\sigma)^2 
\leq (1-c_1\s)^2 \abs{\a}^2+(1-c_2\s)^2\abs{\z}^2
% \\&
=
\abs{\a}^2+\abs{\z}^2-2\s(c_1\abs{\a}^2+c_2\abs{\z}^2)+O(\s^2).
\end{aligned}
\end{equation*}
Thus, we must require that
\begin{equation}
 c_1\abs{\a}^2+c_2\abs{\z}^2>0,
\end{equation} 
where $c_1, c_2$ have been found in the previous step. This inequality is non trivial, because we do not know the sign of $c_1$ and $c_2$ (until Step 3 below).
Inserting 
again the cut relation $\frac{\abs{\z}^2}{\abs{\a}^2}=-\frac{U_\phi S_\phi }{V_\phi }$, we obtain $V_\phi c_1 -  U_\phi S_\phi  c_2>0$. Inserting the expression of $c_2(\phi)$ obtained in \eqref{ciunodue}, we find 
\begin{equation}
\label{ciuno} 
\begin{aligned}
 &(1-S^2_\phi )c_1+2 S_\phi  V_\phi >0 \qquad\text{which gives}
 \\ 
 &(1-S^2_\phi )
 \frac{-2 S_\phi V^3_\phi - U_\phi U'_\phi +U_\phi S_\phi V'_\phi }
 {2U_\phi ^2 - U_\phi  S_\phi V_\phi -V^2_\phi S^2_\phi }+ 2S_\phi V_\phi >0  .
\end{aligned}
  \end{equation}  
%   Inserting the value of $c_1(\phi)$, we obtain the inequality
% \begin{equation}
% % \label{prozio} 
% \begin{aligned}
% \end{aligned}
% \end{equation}
 We check that the latter (strict) inequality holds on the whole interval 
 $[\pi,\phi_1\mathclose[$. By formulas~\eqref{derivouno} for the derivatives   $U', V'$, 
we get  
\begin{equation*}\begin{aligned}
 (1-S_\phi^2) &\Bigl(-2S_\phi V_\phi^3-\frac{\cos\phi}{\phi} U_\phi V_\phi +
 \frac{U_\phi S_\phi^2}{2}-\frac{2U_\phi S_\phi V_\phi}{\phi}\Bigr)                                                                  
 \\&\qquad 
 +2 V_\phi S_\phi\Bigl( 2U_\phi^2- U_\phi V_\phi S_\phi -V_\phi^2S_\phi^2\Bigr)>0.
 \end{aligned}
\end{equation*}
(Recall again that  $-1<S\leq 0$, $\cos \phi <0,$ 
  while $U$ and $ V$ are positive on $[\pi, \phi_1\mathclose[$).
We simplify the term with $V_\phi^3 S_\phi^3$ and we observe that  
\[
\begin{aligned}
-2S_\phi V_\phi^3+(1-S_\phi^2) &\Bigl(-\frac{\cos\phi}{\phi} U_\phi V_\phi +
 \frac{U_\phi S_\phi^2}{2}-\frac{2U_\phi S_\phi V_\phi}{\phi}\Bigr)                                                                  
 +2 V_\phi S_\phi\Bigl( 2U_\phi^2- U_\phi V_\phi S_\phi \Bigr)
  \\&\qquad 
  \gneqq (1-S_\phi^2) \Bigl(
-\frac{2U_\phi S_\phi V_\phi}{\phi}\Bigr)                                                                  
 +2 V_\phi S_\phi\Bigl( 2U_\phi^2- U_\phi V_\phi S_\phi \Bigr)
  \\&\qquad 
  =-2S_\phi U_\phi V_\phi  \Bigl( \frac{(1-S_\phi^2)}{\phi}+V_\phi S_\phi-2U_\phi \Bigr),
  \end{aligned}
\]
 where in the first inequality we  deleted some positive terms  
 (note that $-\frac{\cos\phi}{\phi} U_\phi V_\phi +
 \frac{U_\phi S_\phi^2}{2}>0$ strictly on the interval $[\pi,\phi_1\mathclose[$).  
Next, using the definition of $U, V, S$,  we show that the last parenthesis multiplied by $\phi^3$ is positive.   
\[
\begin{aligned}
\phi^2 \Bigl( {(1-S_\phi^2)}+{\phi}V_\phi S_\phi-2{\phi}U_\phi \Bigr)& =
\phi^2-\sin^2\phi+\frac12 \Bigl(
\sin \phi (\sin \phi- \phi \cos \phi )-\phi (\phi- \sin \phi \cos \phi)
\Bigr)\\
&=
\phi^2-\sin^2\phi+\frac12 \Bigl(
\sin^2 \phi -\phi^2 
\Bigr)  = \frac12 \Bigl(
\phi^2 -\sin^2 \phi 
\Bigr)>0,
 \end{aligned}
\]
as required.

\medskip\noindent\textit{Step 3.} We show that $x'(0)\neq 0$. From  the first line of~\eqref{ciuno},
which has been already proved in Step 2 and from the fact that $S_\phi V_\phi\leq 0$, we conclude  that $c_1(\phi)> 0$    (this and~\eqref{ciunodue} tell us that $c_2<0$). Then, since $x'(0)=(2V_\phi -c_1 S_\phi)\b - c_2\z$, this vector can never 
vanish, because the coefficient $(2V_\phi -c_1(\phi) S_\phi)$ is strictly positive   for all 
$\phi\in[\pi, \phi_1\mathclose[$.  
\end{proof}
%  
% \medskip\noindent\textit{Step 3.}  We show that we may assume that $x'(0)\neq 0$ and  $t'(0)=0$. 
% We use again the invariance of the distance  $d(R\xi,R\tau)= d(\xi,\tau)$, for all $(\xi,\tau)\in \R^3\times\R^3$ and $R\in O(3)$. 
% It suffices to  consider for any $\sigma\geq 0$ a rotation which brings $t(\sigma)$ to the vector $\abs{t(\sigma)}\ol t$ leaving fixed the vector $t(\sigma)\times \ol t$. This rotation can be chosen to depend smoothly from $\sigma$. 
% 
% 
% Therefore we have found a smooth curve $\sigma \mapsto x(\sigma, t(\sigma))$ such that $x'(0)\neq 0$, $t'(0)=0$ and the estimate
% \begin{equation*}
%  d(x(\sigma), t(\sigma))\leq C d(\ol x, \ol t)-C\sigma \quad\text{for all $\sigma \in[0, \sigma_0\mathclose[$.}
% \end{equation*} 

% \medskip\noindent\textit{Step 4.} By previous steps and by invariance of the distance with respect to the rotations  $R(\theta)=\exp(\theta T_{x/\abs{x}})$ living fixed $x$, we improve estimate 

\begin{remark}\label{Heise}
 If we choose $\phi=\pi$ in the proof above, the condition $-\frac{u_\phi s_\phi}{v_\phi}=\frac{\abs{\z}^2}{\abs{\b}^2}$ shows that $\z=0$ and therefore 
$
 (\ol x, \ol t)=(0, U_\pi \a\times\b), 
$ a point on the ``$t$-axis'' of the (Heisenberg) subgroup $\Span\{(\a,0),(\b,0),(0,\a\times\b)\}$. Since $\z=0$, the constant $c_2$ plays no role, the condition on $ x'(0)=2V_\pi\b=\frac{1}{\pi}\b$ becomes empty and the condition on
   $t'(0)=-(2c_1U_\pi + U'_\pi)\a\times\b$   becomes   $ 2c_1 U_\pi + U_\pi' =0$, i.e.~$c_1=\frac{1}{\pi}$. Ultimately,  the curve has the form
\[
\begin{aligned}
 (x(\sigma), & t(\sigma))
 =\Big(S_{\pi-\s}(1-c_1\s) \b,(1-c_1\s)^2U_{\pi-\s}\a\times\b\Big)
 \\=&
 \Bigl(\frac{\sin\s}{\pi}\beta, \frac{1}{4\pi^2}(\pi-\s+\sin\s\cos\s)
 \alpha\times\b\Big)
%  \\&
 =\Big(\frac{\b }{\pi}\s+O(\s^2),\frac{1}{4\pi}(1+O(\s^2))\a\times\b\Big)
\end{aligned}
\]
which gives the corner estimate on the two dimensional surface obtained by rotating the curve $\sigma\mapsto(x(\s), t(\s))$ around the vertical axis $\Span\{\a\times\b\}$ of the Heisenberg subgroup $\Span\{(\a,0),(\b,0),(0,\a\times\b)\}$.
\end{remark}

\begin{proof}[Proof of Corollary \ref{cuttolo}]
 Let $(\ol x, \ol t)=(S_\phi \b+\z,\a\times(U_\phi\b+V_\phi\z))\in \Cut_0$ and let  $\sigma\mapsto(x(\sigma), t(\sigma))=:(x_\s, t_\s)$ be the curve constructed in the proof of the theorem above
 for $\s\geq 0$. Let  $M\in O(3)$ be the rotation of 180 degrees leaving fixed $\ol x$. By invariance of the distance, we have $d(M x_\s ,M t_\s ))=d(x_\s , t_\s )$. Denote by $(\xi_\s, \t_\s)= \frac 12((M x_\s,M t_\s )+   (x_\s , t_\s ))$. Such point belongs to $\Span\{(\ol x,0), (0,\ol x)\}$ (i.e.~to $\Cut_0$) and has coordinates
 \[
  (\xi_\s, \t_\s)=\big(\ol x(1+O(\s^2)), \ol t(1+O(\sigma^2))\big),
 \]
as $\sigma\to 0+$. This follows from the property $\scalar{x'(0)}{\ol x}=\scalar{t'(0)}{\ol t}=0$. By formula \eqref{distanziatore} for the distance on $\Cut_0$,  we have
 \[
  d(\xi_\s,\t_\s)^2= \abs{\xi_\s}^2+R(\phi_\s)\abs{\t_\s}, \quad \text{where }
  \phi_\s=P^{-1}\Bigl(\frac{\abs{\xi_\s}^2}{\abs{\t_\s}}\Bigr).
 \]
Since the function $P^{-1}$ and $R$ are smooth,   we have $\phi_\s=\phi +O(\sigma^2)$ and   $R(\phi_\s)=R(\phi)+O(\s^2)$, which implies  
$  d(\xi_\s,\t_\s)=d(\ol x, \ol t)+O(\sigma^2)$.
Ultimately, we have 
\[
 d(x_\s, t_\s)+ d(M x_\s, M t_\s)-2 d\Big(\frac{ (x_\s, t_\s)+  (M x_\s, M t_\s)}{2}\Big)
 \leq -C \sigma + O(\sigma^2)
\]
and the theorem is proved by letting $\s\to 0+$.
\end{proof}

% 
% 
% Assume that for a given $\phi\in[\pi,\phi_1\mathclose[$ the following three condition are assumed
% \begin{equation}\begin{aligned}
%  (1-S_\phi^2) &\Bigl(-2S_\phi V_\phi^3-\frac{\cos\phi}{\phi} U_\phi V_\phi +
%  \frac{U_\phi S_\phi^2}{2}-\frac{2U_\phi S_\phi V_\phi}{\phi}\Bigr)                                                                  
%  \\&\qquad 
%  +2 V_\phi S_\phi\Bigl( 2U_\phi^2- U_\phi V_\phi S_\phi -V_\phi^2S_\phi^2\Bigr)>0
%  \end{aligned}
% \end{equation} 

\section*{Acknowledgements}  
 We thank Luca Rizzi and Yuri Sachkov who kindly gave us several useful bibliographic information on the state of the art of the subject. We thank the anonymous referee   for raising 
 some useful questions,
which led us to an improvement of the first draft of the paper. 
The authors are members of the {\it Gruppo Nazionale per
l'Analisi Matematica, la Probabilit\`a e le loro Applicazioni} (GNAMPA)
of the {\it Istituto Nazionale di Alta Matematica} (INdAM)

 \footnotesize
% \bibliography{database_16_03}

\newcommand{\etalchar}[1]{$^{#1}$}
\def\cprime{$'$} \def\cprime{$'$}
\providecommand{\bysame}{\leavevmode\hbox to3em{\hrulefill}\thinspace}
\providecommand{\MR}{\relax\ifhmode\unskip\space\fi MR }
\providecommand{\MRhref}[2]{%
  \href{http://www.ams.org/mathscinet-getitem?mr=#1}{#2}
}
\providecommand{\href}[2]{#2}

%\bibliography{C:\Users\Utente\Dropbox\Matematica\Bibtex\database_16_03}
%\bibliography{/home/daniele/Dropbox/Matematica/Bibtex/database_16_03}

% \bibliographystyle{amsalpha}
\normalsize
\normalsize
\bigskip \noindent\sc \small  Annamaria Montanari, Daniele Morbidelli
\\ Dipartimento di Matematica,
Universit\`{a} di Bologna  (Italy)
\\Email: \tt   annamaria.montanari@unibo.it,
daniele.morbidelli@unibo.it

\end{document}